\documentclass[12pt,a4paper]{article}
\usepackage{amsmath}
\usepackage{amssymb}
\usepackage{graphicx}
\usepackage{xcolor}
\usepackage{amsthm}
\usepackage{color}

\theoremstyle{plain}
\begingroup
\newtheorem{theorem}{Theorem}[section]

\newtheorem{proposition}[theorem]{Proposition}

\endgroup

\theoremstyle{definition}
\begingroup
\newtheorem{definition}[theorem]{Definition}
\newtheorem{remark}[theorem]{Remark}

\endgroup

\title{An integral-representation result for continuum limits of discrete energies with multi-body interactions}
\author{Andrea Braides\\Dipartimento di Matematica, Universit\`a di Roma Tor Vergata
\\ via della ricerca scientifica 1, 00133 Roma, Italy\\ \small
{\tt e-mail:~braides@mat.uniroma2.it}\\
\\ Leonard Kreutz\\ 
Gran Sasso Science Institute, Viale F.~Crispi 7, 67100 L'Aquila, Italy \\ 
\small{\tt e-mail:~leonard.kreutz@gmail.com}
}
\date{}                                           

\usepackage[english]{babel}

\begin{document}

\def\e{\varepsilon}
\def\dx{\mathrm{d}x}
\def\ZZ{{\mathbb Z}}
\def\RR{{\mathbb R}}
\maketitle

\begin{abstract}
We prove a compactness and integral-representation theorem for sequences of families of lattice energies
describing atomistic interactions defined on lattices with vanishing lattice spacing.
The densities of these energies may depend on interactions between all points of the corresponding lattice contained in a reference set. We give conditions that ensure that the limit is an integral defined on a Sobolev space. A homogenization theorem is also proved. The result is applied to multibody interactions corresponding 
to discrete Jacobian determinants and to linearizations of Lennard-Jones energies with mixtures of convex and concave quadratic pair-potentials.
\end{abstract}


\section{Introduction}
This paper focuses on the passage from lattice theories to continuum ones in 
the framework of variational problems, such as for atomistic systems in Computational Materials Science (see e.g.~\cite{BLBL2}). For notational convenience we will state our results for energies 
defined on functions $u$ parameterized on a portion of $\mathbb{Z}^N$ (with values in $\mathbb{R}^n$),
but our assumptions may be immediately extended to more general lattices. For central interactions
such energies may be written as 
\begin{equation}\label{Enn}
E(u)= \sum_{i,j} \psi_{ij} (u_i-u_j),
\end{equation}
where $i,j$ are points in the domain under consideration. We are interested in the behaviour
of such an energy when the dimensions of the domain are much larger than the lattice spacing.
In the discrete-to-continuum approach this can be done by approximation with a continuum energy
obtained as a limit after a scaling argument. To that end, we introduce a small parameter $\e$ (which, for the
unscaled energy $E$ is the inverse of the linear dimension of the domain) and scale the energies as
\begin{equation}\label{Enne}
E_\e(u)= \sum_{i,j}\e^N \psi^\e_{ij} \Bigl({u_i-u_j\over\e}\Bigr),
\end{equation}
where now $i,j$ belong to a domain $\Omega$ independent of $\e$, and the domain of $u$ is
$\Omega\cap\e\mathbb{Z}^N$; accordingly, we set $\psi^\e_{ij}=\psi_{i/\e\,j/\e}$.
Both scalings, $\e^N$ of the energy, and $u_i/\e$ of the function, 
are important in this process and highlight that in this case we are regarding the energy as a volume 
integral ($\e^N$ being the volume element of a lattice cell) depending on a gradient (${(u_i-u_j)/\e}$
being interpreted as a scaled difference quotient or discrete gradient). Other scalings are possible 
and give rise to different types of energies, depending on the form of $\psi^\e_{ij}$, highlighting
the multiscale nature of the problem. In the present context we focus on this particular ``bulk''
scaling (for an account of other scaling limits see \cite{Hand,B-Seoul}).

The continuum approximation of $E_\e$ is obtained by taking a limit as $\e\to 0$. This has been done in different ways, using a pointwise limit in \cite{BLBL} (where lattice functions are considered as restrictions of a smooth function to $\ZZ^N$) or a $\Gamma$-limit in \cite{AC} (in this case lattice functions
are extended as piecewise-constant functions and embedded in some common Lebesgue space) to obtain an energy of the form
\begin{equation}\label{Fnn}
F(u)= \int_\Omega f(x,\nabla u)\dx
\end{equation}
with domain a Sobolev space.
We focus on the result of \cite{AC}, which relies on the localization methods of $\Gamma$-convergence
(see \cite{GCB} Chapter 12) envisaged by De Giorgi to deduce the integral form of the $\Gamma$-limit from its behaviour both as a function of $u$ and $\Omega$.
Conditions that allow to apply those methods are

(i) ({\em coerciveness}) growth conditions from below that allow to deduce that the limit is defined on some Sobolev space; e.g.~that $\psi^\e_{ij}(w)\ge c(|w|^p-1)$ for nearest-neighbours and $\psi^\e_{ij}\ge0$ for all $i,j$;

(ii) ({\em finiteness})  growth conditions from above that allow to deduce that the limit is finite on the same Sobolev space;
e.g.~that $\psi^\e_{ij}(w)\le c^\e_{ij}(|w|^p+1)$ for all $ij$, with some summability conditions on $c^\e_{ij}$
uniformly in $\e$;

(iii) ({\em vanishing non-locality}) conditions that allow to deduce that the $\Gamma$-limit is a measure in its dependence on $\Omega$. This is again obtained from some uniform decay conditions on the coefficients $c^\e_{ij}$.

Hypotheses (i)--(iii) are sharp, in the sense that failure of any of these conditions may result in a $\Gamma$-limit that cannot be represented as in (\ref{Fnn}). The result in \cite{AC} has been successful in many applications, 
among which the computation of optimal bounds for conducting networks \cite{BF}, the derivation of nonlinear elastic energies from atomistic systems \cite{AC,KLR}, of their linear counterpart \cite{BSV}, and of $Q$-tensor theories from spin interactions \cite{BCS}, numerical homogenization \cite{Gloria}, the analysis of the pile-up of dislocations \cite{Geers}, and others. Moreover, it has been extended to cover stochastic lattices \cite{ACG} and dimension-reduction problems \cite{ABC}. However, its range of applicability is restricted to pairwise interactions, which implies constraints on the possible energy densities. The main motivation of the present work is to overcome some of those limitations. More precisely, we focus on two issues:

$\bullet$ {\em the extension to the result to many-body interactions}. In principle, a point in the lattice may interact with all other points in the domain $\Omega$. As a particular case, we may think of $k$-body interactions corresponding to the minors of the lattice transformation (which is affine at the lattice level), such as the discrete determinant in two dimensions, which can be viewed as a three-point interaction. Some works in this direction are already present in the literature for particular cases \cite{BS,LR,MP};

$\bullet$ {\em the use of averaged growth conditions on the energy densities}. Some lattice energies are obtained as an approximation of non-convex long-range interactions. As such, even when considering pair interactions, they may fail to satisfy coerciveness conditions for some $\psi_{ij}$. As an example we can think of the linearization of Lennard-Jones interactions, which gives concave quadratic energies for distant $i$ and $j$.
The coerciveness of the energy can nevertheless be recovered using the fast decay of the potential so that
short-range convex interactions dominate long-range concave ones. In general, coerciveness can be obtained by substituting a growth conditions on each of the interactions with an averaged  growth condition.

\smallskip
In order to achieve the greatest generality, we assume that energy densities may indeed depend on all points in $\Omega\cap\e\ZZ^N$. An energy density $\phi^\e_i$ will describe the interaction of a point $i\in \Omega\cap\e\ZZ^N$ with all other points in the domain. This standpoint, already used in \cite{BC} for surface energies
in a simpler setting (see also \cite{BLO} in a one-dimensional setting), brings some notational complications (except for the case $\Omega=\RR^N$) since it is convenient to regard each such function as defined on a different set $(\Omega-i)\cap\e\ZZ^N$. This complication is anyhow present each time that we consider more-than-two-body interactions. The energies are then defined as
\begin{align}\label{Feppa}
F_\varepsilon(u) = \sum_{i \in \Omega\cap\e\ZZ^N} \varepsilon^N \phi_i^\varepsilon(\{u_{j+i}\}_{j \in (\Omega-i)\cap\e\ZZ^N}).
\end{align}
An important remark to make is that there are many ways to define energy densities giving the same $F_\e$.
Note for example that for central interactions as above $\phi_i^\varepsilon$ may be simply given by
\begin{equation}\label{cetra}
\phi_i^\varepsilon(\{z_j\})=\sum_{j\in  (\Omega-i)\cap\e\ZZ^N} \psi^\e_{ij}\Bigl({z_j-z_0\over \e}\Bigr)
=\sum_{j\in  (\Omega-i)\cap\e\ZZ^N} \psi_{i/\e\,j/\e}\Bigl({z_j-z_0\over \e}\Bigr),
\end{equation}
but the interactions may also be regrouped differently and in principle $\phi_i^\varepsilon$ may 
include some  $\psi^\e_{kj}$ with $k\neq i$. This is important in order to allow that some $\psi^\e_{ij}$ be unbounded from below, up to satisfying a lower bound when considered together with the other interactions.

The set of hypotheses we are going to list for $\phi^\e_{ij}$ will allow to treat a larger class of energies than those of the form (\ref{Enne}), but they must be stated with some care. The precise statements are given in Section \ref{MT}. Here we give a simplified description as follows:

(o) ({\em translational invariance in the codomain}) $\phi_i^\varepsilon(\{z_j+w\})= \phi_i^\varepsilon(\{z_j\})$
for all $i$, $\{z_j\}$ and vector $w$. This condition is automatically satisfied for interactions depending on differences $z_i-z_j$;

(i) ({\em coerciveness}) the energy must be estimated from below by a nearest-neighbour pair energy
and $\phi^\e_{i}\ge0$ for all $i$. This condition is less restrictive than the corresponding one for pair interactions since it refers to an already averaged energy density;

(ii) ({\em Cauchy-Born hypothesis}) we assume a polynomial upper bound for $F_\e(u)$ only when $u$ is linear.
For energy densities as in (\ref{cetra}) this in general rewritten in terms of $\psi_{ij}$ as 
\begin{equation}
\Psi(M):=\sum_{j} \psi_{i\,i+j}(Mj)\le C(1+|M|^p),
\end{equation}
for all $i\in \ZZ^N$, and all $n\times N$ matrices $M$.
This condition is in principle weaker than the finiteness property (ii) for pair interactions. Examining this condition separately goes in the direction of analyzing first pointwise convergence (as in \cite{BLBL}) and then $\Gamma$-convergence;

(iii) ({\em vanishing non-locality}) we assume that if $u=v$ on a square of centre $i$ and side-length $\delta$ then 
$$
\phi_i^\varepsilon(\{u_{j+i}\}_{j \in (\Omega-i)\cap\e\ZZ^N})\leq
\phi_i^\varepsilon(\{v_{j+i}\}_{j \in (\Omega-i)\cap\e\ZZ^N})+r(\e,\delta,\|\nabla u\|_p)
$$
($u$ is identified with a piecewise-affine interpolation), where the rest $r$ is negligible as $\e\to0$ for $\|\nabla u\|_p$ bounded.
Note that this condition is automatically satisfied with $r=0$ if the range of the interactions is finite,
and can be deduced from the corresponding condition (iii) for central interactions;

(iv) ({\em controlled non-convexity}) a final condition must be added to ensure that the limit be a measure as a function of $\Omega$.  For central interactions, this condition is hidden in the previous (i) and (ii), which imply a 
convex growth condition on $\Psi$; more precisely a polynomial growth of the form
$$
c(|M|^p-1)\le \Psi(M)\le C(1+|M|^p).
$$
This double inequality allows to use classical convex-combination arguments with cut-off functions even though $\Psi$ may not be convex. In our case this compatibility with convex arguments must be required separately, and is formalized in condition (H5) in Section \ref{HED}.

Under the conditions above we again deduce that $\Gamma$-limits of energies $F_\e$ are integral functionals $F$ as in (\ref{Fnn}) defined on a Sobolev space. The integrand $f$ can be described by a derivation formula, which is allowed by the study of suitably defined boundary-value problems. This derivation formula can also be used to prove a periodic-homogenization result. In the generality of energies possibly depending on the interaction of all points in $\Omega$ some care must be used to define periodicity for the energy densities. In the case of finite-range interactions we require that in the interior of $\Omega$ we have $\phi^\e_i=\phi_{\e/i}$, where $\phi_k$ 
is periodic in $k$. For infinite-range interactions the definition is given by approximation with periodic energy densities with  finite-range interactions.

\smallskip
The paper is organized as follows. After some notation, in Section \ref{MT} we rigorously state the hypotheses outlined above and prove the main compactness and integral-representation theorem. Section \ref{DBT} is devoted to formalizing and proving the convergence of Dirichlet boundary-value problems, which is used in the following Section \ref{HOM} to state and derive a homogenization formula. Finally, Section \ref{EXA} is devoted to examples. More precisely, we show how our hypotheses are satisfied by functions depending on discrete determinants and by a linearization of Lennard-Jones energies mixing convex and concave quadratic pair energy densities.
Finally, in the same section we recover the result in \cite{AC} as a particular case of our main theorem.

\section{Notation and preliminaries}
We denote by $\Omega$ an open and bounded subset of $\mathbb{R}^N$ with Lipschitz boundary. We set $Q$ to be the unit cube with sides orthogonal to the canonical orthonormal basis $\{e_1,\ldots,e_N\}$, $Q=\{x \in \mathbb{R}^N : |\langle x, e_i \rangle| \leq \frac{1}{2}, \text{ for all } i =1,\ldots,N\}$ and for $\delta>0$ we define $Q_\delta=\delta Q$. Moreover, for $x \in \mathbb{R}^N$ we set $Q(x) =Q +x$ and $Q_\delta(x) = Q_\delta + x$.
 We set $\mathcal{A}(\Omega) =\{A \subset \Omega : A \text{ open}\}$, $\mathcal{A}^{reg}(\Omega) = \{ A \in \mathcal{A}(\Omega) : \partial A \text{ Lipschitz}\}$, and for $\delta >0$ set
$ A_\delta = \{x \in \Omega : \mathrm{dist}_\infty(x,A) < \delta\} $ and $A^\delta = \{x \in A : \mathrm{dist}_\infty(x,A^c) > \delta\} $. For $B \subset \mathbb{R}^N$ we write $|B| $ for the $N$-dimensional Lebesgue measure of $B$. For a vector $x \in \mathbb{R}^N$ we set
\begin{align*}
\lfloor x\rfloor = (\lfloor x_1\rfloor, \ldots, \lfloor x_N \rfloor).
\end{align*}
We define for $u : \mathbb{R}^N \to \mathbb{R}^n, \xi \in \mathbb{Z}^N, x \in \mathbb{R}^N$ and $\varepsilon >0$
\begin{align*}
D^\xi_\varepsilon u(x) := \frac{u(x+\varepsilon\xi)-u(x)}{\varepsilon|\xi|}
\end{align*}
the discrete difference quotient of $u$ at $x$ in direction $\xi$. 

For a function $u$ we set $C(u)$ to be a constant depending on $u$, the dimension and its domain of definition and which may vary from line to line.

\smallskip\noindent
{\bf Slicing.}
We recall the standard notation for slicing arguments (see \cite{AFP}). Let $\xi \in S^{N-1}$, and let $\Pi_\xi = \{y \in \mathbb{R}^N : \langle y, \xi \rangle = 0\}$ be the linear hyperplane orthogonal to $\xi$. If $y \in \Pi_\xi$ and $E \subset \mathbb{R}^N$ we define $E_\xi =\{y \in \Pi_\xi \text{ such that }\exists t \in \mathbb{R} : y +t\xi\in E\}$ and $E_y^\xi =\{t\in \mathbb{R} : y+t\xi \in E\}$. Moreover, if $u:E \to \mathbb{R}^n$ we set $u_{\xi,y} : E_y^\xi  \to \mathbb{R}^n$ to $u_{\xi,y}(t) = u(y + t\xi)$.

\smallskip\noindent{\bf $\Gamma$-convergence.}
A sequence of functionals $F_n : L^p(\Omega;\mathbb{R}^n) \to [0,+\infty]$ is said to $\Gamma$-converge to a functional $F : L^p(\Omega;\mathbb{R}^n) \to [0,+\infty]$ at $u \in L^p(\Omega;\mathbb{R}^n)$ as $n \to \infty$ and we write $\displaystyle F(u) = \Gamma\text{-}\lim_{n \to \infty} F_n(u)$ if the following two conditions are satisfied:
\begin{itemize}
\item[(i)] For every $u_n $ converging to $ u$ in $L^p(\Omega;\mathbb{R}^n)$ we have
$\displaystyle
\liminf_{n \to \infty} F_n(u_n) \geq F(u).
$
\item[(ii)] There exists a sequence $\{u_n\}_n \subset L^p(\Omega;\mathbb{R}^n)$ converging to $u$ in $L^p(\Omega;\mathbb{R}^n)$ such that $\displaystyle
\limsup_{n \to \infty} F_n(u_n) \leq F(u).
$
\end{itemize}
We say that $F_n$ $\Gamma$-converges to $F$ if $\displaystyle F(u) = \Gamma\text{-}\lim_{n\to \infty}F_n(u)$ for all $u \in L^p(\Omega;\mathbb{R}^n)$. 

If $\{F_\varepsilon\}_{\varepsilon >0}$ is a family of functionals indexed by a continuous parameter $\varepsilon >0$ we say that $F_\varepsilon$ $\Gamma$-converges to $F$ as $\varepsilon \to 0^+$ if for all $\varepsilon_n \to 0$ we have that $F_{\varepsilon_n}$ $\Gamma$-converges to $F$. We define the $\Gamma$-$\liminf$ $F' : L^p(\Omega;\mathbb{R}^n) \to [0,\infty]$ and the $\Gamma$-$\limsup$ $F'' : L^p(\Omega;\mathbb{R}^n) \to [0,\infty]$ respectively by 
\begin{align*}
&F'(u)= \Gamma\text{-}\liminf_{\varepsilon \to 0} F_\varepsilon(u) = \inf \Big\{\liminf_{\varepsilon \to 0} F_\varepsilon(u_\varepsilon) : u_\varepsilon \to u\Big\}, \\&F''(u)= \Gamma\text{-}\limsup_{\varepsilon \to 0} F_\varepsilon(u) = \inf \Big\{\limsup_{\varepsilon \to 0} F_\varepsilon(u_\varepsilon) : u_\varepsilon \to u\Big\}.
\end{align*}
Note that the functionals $F'$,$F''$ are lower semicontinuous and $F_\varepsilon$ $\Gamma$-converges to $F$ as $\varepsilon \to 0^+$ if and only if $F =F'=F''$.

\smallskip\noindent
{\bf Lattice functions.} 
For  $A \in \mathcal{A}(\Omega), $ we set  $Z_\varepsilon(A) = \varepsilon\mathbb{Z}^N \cap A$
We set $\mathcal{A}_\varepsilon(\Omega,\mathbb{R}^n) :=\left\{ u :Z_\varepsilon(\Omega) \to \mathbb{R}^n\right\} $. 
\begin{definition}(Convergence of discrete functions) Functions $u \in \mathcal{A}_\varepsilon(\Omega;\mathbb{R}^n)$ can be interpreted by functions belonging to the space $L^p(\Omega;\mathbb{R}^n)$ by setting (with slight abuse of notation) $u(z)=0 $ for all $z \in Z_\varepsilon(\Omega^c)$ and
\begin{align*}
u(x) = u(z^\varepsilon_x)
\end{align*}
where $z^\varepsilon_x$ is the closest point of $Z_\varepsilon(\mathbb{R}^N)$ to $x$ (which is uniquely defined up to a set of measure $0$). We then say that $u_\varepsilon \to u$ in $L^p(\Omega;\mathbb{R}^n)$ if the interpolations of $u_\varepsilon$ converge to $u$ in $L^p(\Omega;\mathbb{R}^n)$.
\end{definition}

\smallskip\noindent
{\bf Integral representation.} We will use the following integral representation result (see \cite{BDF}).
\begin{theorem}\label{Theorem1} Let $F : W^{1,p}(\Omega;\mathbb{R}^n) \times \mathcal{A}(\Omega) \to [0,+\infty]$ satisfy the following properties
\begin{itemize}
\item[i)](measure property) For every $u \in W^{1,p}(\Omega;\mathbb{R}^n)$ we have that $F(u,\cdot)$ is the restriction of a Radon measure to the open sets. 
\item[ii)] (lower semicontinuity) For every $A \in \mathcal{A}(\Omega)$ we have that $F(\cdot,A)$ is weakly-$W^{1,1}(\Omega;\mathbb{R}^n)$ lower semicontinuous.
\item[iii)] (bounds) For every $(u,A) \in W^{1,p}(\Omega;\mathbb{R}^n) \times \mathcal{A}(\Omega)$ it holds that
\begin{equation*}
0 \leq F(u,A) \leq C\left(\int_{A} \left|\nabla u\right|^p\mathrm{d}x + |A| \right)
\end{equation*}
\item[iv)](translational invariance) For every $(u,A) \in W^{1,p}(\Omega;\mathbb{R}^n)\times \mathcal{A}(\Omega)$ and for every $c \in \mathbb{R}^n$ it holds $F(u,A) = F(u + c,A)$.
\item[v)](locality) For every $A \in \mathcal{A}(\Omega)$ and every $u,v \in W^{1,p}(\Omega;\mathbb{R}^n)$ such that $u=v$ a.e. in $A$, we have that $F(u,A)=F(v,A)$.
\end{itemize}
Then there exists a Carath\` eodory function $f: \Omega \times \mathbb{R}^{n \times N} \to [0,+\infty]$ such that 
\begin{equation*}
F(u,A) = \int_A f(x,\nabla u) \mathrm{d}x
\end{equation*}
for every $(u,A) \in W^{1,p}(\Omega;\mathbb{R}^n) \times \mathcal{A}(\Omega)$. 
\begin{itemize}
\item[vi)] (translational invariance in $x$) if for every $M \in \mathbb{R}^{n\times N},$ $z,y \in \Omega$ and for every $\rho >0$ such that $Q_\rho(z) \cup Q_\rho(y) \subset \Omega$ we have that
\begin{align*}
F(Mx,Q_\rho(y)) = F(Mx,Q_\rho(z)),
\end{align*}
then $f$ does not depend on $x$.
\end{itemize}

\end{theorem} 

\section{The main result}\label{MT}
For all $i\in\Omega$, we denote by $\Omega_i = \Omega - i$ the translation of the set $\Omega$ with $i$ at the origin, and we consider a function $\phi_i^\varepsilon :\left(\mathbb{R}^n\right)^{Z_\varepsilon(\Omega_i)} \to [0,+\infty)$. Let $F_\varepsilon : \mathcal{A}_\varepsilon(\Omega,\mathbb{R}^n) \times\mathcal{A}(\Omega) \to [0,+\infty)$ be defined by
\begin{align}\label{DefEnergy}
F_\varepsilon(u,A) = \sum_{i \in Z_\varepsilon(A)} \varepsilon^N \phi_i^\varepsilon(\{u_{j+i}\}_{j \in Z_\varepsilon (\Omega_i)}).
\end{align}
In this section we give hypothesis on the energy densities $\phi_i^\varepsilon$ in order to ensure that the $\Gamma$-limit of the energies defined in (\ref{DefEnergy}) be finite only on $W^{1,p}(A,\mathbb{R}^n) \cap L^p(\Omega;\mathbb{R}^n)$ and there exists a Carath\'eodory function $f : \Omega \times \mathbb{R}^{n\times N} \to [0,\infty)$ such that
\begin{align}\label{Gradientlimit}
F(u,A) = \int_{A}f(x,\nabla u(x))\mathrm{d}x
\end{align}
for all $(u,A) \in W^{1,p}(A,\mathbb{R}^n) \cap L^p(\Omega;\mathbb{R}^n)\times \mathcal{A}(\Omega)$.
A corresponding problem on the continuum is one of the first formalized in the theory of $\Gamma$-convergence, when $F_\varepsilon$ themselves are integral energies. In that approach integral functionals are interpreted as depending on a pair $(u,A)$ with u a Sobolev function and $A$ a subset of $\Omega$, when the integration is performed on $A$ only. The compactness property of $\Gamma$-convergence then ensures that a $\Gamma$-converging subsequence exits on a dense family of open sets by a simple diagonal argument. Showing that the dependence of the limit on the set variable is that of a regular measure, the convergence is extended to a larger family of sets, and an integral representation result can be applied. The type of conditions singled out in that case can be adapted to the discrete setting, taking into account that discrete energies are ``nonlocal'' in nature since they depend on the interactions of points at a finite distance. The locality of the limit energy $F$ must then be assured by a requirement of ``vanishing nonlocality'' as $\varepsilon\to 0$.

\subsection{Hypotheses on the energy densities}\label{HED}
A first requirement is that $F_\varepsilon$ be invariant under addition of constants to $u$; namely

\smallskip 

(H1) ({\em translational invariance}) for all $w\in \mathbb{R}^n$ we have
\begin{equation}
 \phi^\varepsilon_i(\{z_{j}+w\}_{j \in Z_\varepsilon(\Omega_i)})= \phi^\varepsilon_i(\{z_{j}\}_{j \in Z_\varepsilon(\Omega_i)})
\end{equation}
for all $\varepsilon>0$, $i\in Z_\varepsilon(\Omega)$ and $z:Z_\varepsilon(\Omega)\to\mathbb{R}^n$.

\smallskip

A second requirement is that $F_\varepsilon(u_\varepsilon)$ be finite if $\widehat u_\varepsilon$ are a discretization of a $W^{1,p}$ function. In particular this should hold for affine functions.

\smallskip

(H2) ({\em upper bound for the Cauchy-Born approximation}) there exists $C >0$, such that for every $M \in \mathbb{R}^{n\times N}$ and $M x (i) = M i$ we have
\begin{align} \label{Cauchy-Born}
\phi_i^\varepsilon(\{(M x)_{j}\}_{j \in Z_\varepsilon(\Omega_i)}) \leq C(|M|^p+1)
\end{align}
for all $\varepsilon >0$ and all $i \in Z_\varepsilon(\Omega)$.

\smallskip

We then also require that the limit domain be exactly $W^{1,p}$ functions, with $p>1$.
To that end a coerciveness condition should be imposed. 

\smallskip
(H3)  ({\em equi-coerciveness})  there exists $c>0$ such that
\begin{equation}
c\Bigl(\sum_{n=1}^N\left|D^{e_n}_\varepsilon z(0)\right|^p -1\Bigr)\le \phi^\varepsilon_i(\{z_{j}\}_{j \in Z_\varepsilon(\Omega_i)})
\end{equation}
for all $\varepsilon$ and $i$ such that $i+e_n\in Z_\varepsilon(\Omega)$ for all $n \in \{1,\cdots, N\}$.

\smallskip

Next, we have to impose that the approximating continuum energy be local. Indeed,
in principle discrete interactions are non-local, in that they take into account nodes of the lattice at a finite distance. This condition ensures that we can always find recovery sequences for a set $A \in \mathcal{A}(\Omega)$ that will not oscillate too much a finite distance away from $A$. We expect the limit to depend on $\nabla u$ if only the interactions for small distances are relevant, or, in other words, if the decay of interactions is fast enough. This can be formulated otherwise: we may require that the overall effect of long-range interactions at a point decay sufficiently fast as follows.

\smallskip

(H4) ({\em decaying non-locality})  There exist $\{C_{\varepsilon,\delta}^{j,\xi}\}_{\varepsilon >0,\delta>0,j \in \varepsilon\mathbb{Z}^N, \xi \in \mathbb{Z}^N}$, $C^{j,\xi}_{\varepsilon,\delta} \geq 0$ satisfying 
\begin{align}\label{Assumptions on Cj}
 \limsup_{\varepsilon \to 0}  \sum_{j \in Z_\varepsilon(\mathbb{R}^N),\xi \in \mathbb{Z}^N} C^{j,\xi}_{\varepsilon,\delta} =0 \quad \forall \delta >0
\end{align}
such that for all $\delta >0$, $z,w \in \mathcal{A}_\varepsilon(\Omega,\mathbb{R}^n)$ satisfying $z(j) =w(j) $ for all $j \in Z_\varepsilon(Q_\delta(i))$ we have 
\begin{align*}
\phi_i^\varepsilon(\{z_{j}\}_{j \in Z_\varepsilon(\Omega_i)}) \leq &  \phi_i^\varepsilon(\{w_{j}\}_{j \in Z_\varepsilon(\Omega)}) + \underset{j+\varepsilon\xi \in Z_\varepsilon(\Omega_i)}{\sum_{j \in Z_\varepsilon(\Omega_i),\xi \in \mathbb{Z}^N}}C^{j,\xi}_{\varepsilon,\delta} \left(|D^\xi_\varepsilon z(j)|^p +1\right)  .
\end{align*}

The final condition is the most technical and derives from our requirement that the limit can be expressed in terms of an integral. This is the most restrictive in the vectorial case $d >1$ where convexity conditions have to be relaxed. A function $\psi : Z_\varepsilon(\Omega) \to \mathbb{R}$ is called a {\em cut-off function} if $0 \leq \psi \leq 1$.

\smallskip

(H5) ({\em controlled non-convexity}) There exist $C>0$ and $\{C_{\varepsilon}^{j,\xi}\}_{\varepsilon >0,j \in \varepsilon\mathbb{Z}^N, \xi \in \mathbb{Z}^N}$, $C^{j,\xi}_{\varepsilon} \geq 0$ satisfying 
\begin{align}\label{Assumptions on Cxi}
\limsup_{\varepsilon \to 0} \sum_{j \in Z_\varepsilon (\mathbb{R}^N),\xi \in \mathbb{Z}^N} C^{j,\xi}_{\varepsilon} <+\infty,\quad \forall \, \delta >0 \text{ we have } \limsup_{\varepsilon \to 0}  \sum_{\max\{\varepsilon|\xi|,|j|\} >\delta} C^{j,\xi}_{\varepsilon} =0
\end{align} 
such that for all $z,w \in \mathcal{A}_\varepsilon(\Omega,\mathbb{R}^n)$ 
 and $\psi$ cut-off functions we have
\begin{align*}
\phi_i^\varepsilon(\{\psi_{j}z_{j} + (1-\psi_{j})w_{j}\}_{j \in Z_\varepsilon(\Omega_i)}) \leq & C \left( \phi_i^\varepsilon( \{z_{j}\}_{j \in Z_\varepsilon(\Omega_i)}) + \phi_i^\varepsilon( \{w_{j}\}_{j \in Z_\varepsilon(\Omega_i)})\right) \\&+R^\varepsilon_i(z,w,\psi)
\end{align*}
where
\begin{align*}
R^\varepsilon_i(z,w,\psi) = &\underset{j+\varepsilon\xi \in Z_\varepsilon(\Omega_i)} {\sum_{j \in Z_\varepsilon(\Omega_i),\xi \in \mathbb{Z}^N}} C_{\varepsilon}^{j,\xi} \Big((\underset{n \in \{1,\ldots,N\}}{\sup_{k \in Z_\varepsilon(\Omega_i)}} |D^{e_n}_\varepsilon \psi(k)|^p+1) |z(j+\varepsilon\xi)-w(j+\varepsilon\xi)|^p\Big) \\+&\underset{j+\varepsilon\xi \in Z_\varepsilon(\Omega_i)} {\sum_{j \in Z_\varepsilon(\Omega_i),\xi \in \mathbb{Z}^N}} C_{\varepsilon}^{j,\xi}\Big(  |D^\xi_\varepsilon z(j)|^p + |D^\xi_\varepsilon w(j)|^p +1 \Big).
\end{align*}

\medskip

\begin{remark}{(observations on the assumptions)} If condition (H1) fails we expect the limit not to be translational invariant anymore and if a integral representation exists it is expected to be of the form
\begin{align*}
F(u,A) = \int_A f(x,u,\nabla u)\mathrm{d}x.
\end{align*}
However, integral-representation theorems for non-translation-invariant functionals in general require restrictive hypotheses that should be added to (H2)--(H5). 

If condition (H2) fails the $\Gamma$-limit may not be finite on $W^{1,p}(\Omega;\mathbb{R}^n)$. Condition (H3) allows to estimate nearest-neighbour interactions centered in $i$ in terms of $\phi_i^\varepsilon$. Note that this estimate may still be true even if there are no interactions of the type $|D^{e_n}_\varepsilon u|^p$ taken into account by $\phi_i^\varepsilon$. Indeed if $d=1$ we may take $c_2,c_3 >0$
\begin{align*}
\phi_i^\varepsilon(\{z_{j}\}_{j \in Z_\varepsilon(\Omega_i)}) = c_2 \left|\frac{z_3-z_1}{2\varepsilon}\right|^2 + c_3 \left|\frac{z_3-z_0}{3\varepsilon}\right|^2.
\end{align*}
If we assume a finite range $R$ of interactions and assume that the potential $\phi_i^\varepsilon$ is well behaved in some sense condition (H4) is always satisfied and in the definition of $R_i^\varepsilon$ the summation is only taken over $Q_R(i)$.
If condition (H4) fails the $\Gamma$-limit may be non-local. Indeed there are examples where functionals of the form
\begin{align*}
F(u) = \int_\Omega |\nabla u|^2 \mathrm{d}x + \int_{\Omega \times \Omega} k(x,y)|u(x)-u(y)|^2\mathrm{d}x
\end{align*}
can be obtained as the $\Gamma$-limit of energies of the form
\begin{align*}
F_\varepsilon(u) = \sum_{i \in Z_\varepsilon(\Omega)} \underset{i+\varepsilon\xi \in Z_\varepsilon(\Omega)}{\sum_{\xi \in \mathbb{Z}^N}} \varepsilon^N c^\varepsilon_{i,\xi} |D^\xi_\varepsilon u(i)|^2.
\end{align*}
Note that (H1) is still satisfied. Condition (H5) mimics the so-called \textit{fundamental estimate} in the continuum and ensures that the limit $F(u,\cdot)$ be subadditive as a set function. Note that this condition is satisfied for subadditive potentials with appropriate growth conditions. In particular, in Section \ref{ACt} we show how 
the hypotheses above can be deduced from those in \cite{AC} in the case of pair potentials.
\end{remark}
\subsection{Compactness and integral representation}
The goal of this section is to establish the proof of Theorem \ref{Compactness}.
\begin{theorem}{\em(\bf{Integral Representation})}\label{Compactness} Let $F_\varepsilon : L^p(\Omega;\mathbb{R}^n) \to [0,+\infty]$ be defined by {\rm(\ref{DefEnergy})}, where $\phi_i^\varepsilon : (\mathbb{R}^n)^{Z_\varepsilon(\Omega)} \to [0,+\infty)$ satisfy {\rm(H1)--(H5)}. Then for every sequence $(\varepsilon_j)$ of positive numbers converging to $0$, there exists a subsequence $\varepsilon_{j_k}$ and a Carath\'eodory function $f: \Omega \times \mathbb{R}^{n\times N} \to [0,+\infty)$, quasiconvex in the second variable satisfying
\begin{align*}
c(|\xi|^p-1) \leq f(x,\xi) \leq C(|\xi|^p+1)
\end{align*}
with $0<c < C$, such that $F_{\varepsilon_{j_k}}(\cdot)$ $\Gamma$-converges with respect to the $L^p(\Omega;\mathbb{R}^n)$-topology to the functional $F: L^p(\Omega;\mathbb{R}^n) \to [0,+\infty]$ defined by
\begin{align*}
F(u) = \begin{cases} \displaystyle\int_{\Omega} f(x,\nabla u)\mathrm{d}x &\text{if }u \in W^{1,p}(\Omega;\mathbb{R}^n)\\
+\infty &\text{otherwise.}
\end{cases}
\end{align*}
Moreover, for any $u \in W^{1,p}(\Omega;\mathbb{R}^n)$ and any $A \in \mathcal{A}(\Omega)$ we have
\begin{align*}
\Gamma\text{-}\lim_{k \to +\infty} F_{\varepsilon_{j_k}}(u,A) = \int_A f(x,\nabla u)\mathrm{d}x.
\end{align*}
\end{theorem}

We will derive the proof of Theorem \ref{Compactness} as a consequence of some propositions and lemmas, which are fundamental in order to show that our limit functionals satisfies all the assumption of Theorem \ref{Theorem1}. In the next two proposition we show with the use of (H1)--(H5) that assumptions (ii) and (iii) of Theorem \ref{Theorem1}  are satisfied. Note that property (\ref{lowerbound}) below allows to deduce weak lower-semicontinuity in $W^{1,p}$ even though we prove the $\Gamma$-convergence of the discrete energies with respect to the strong $L^p(\Omega;\mathbb{R}^n)$-topology, so that assumption (ii) is satisfied.
\medskip

Note that the proof of Proposition 3.3 is the same as the proof of Proposition 3.4 in \cite{AC}. We repeat it here only for completeness and the reader's convenience.

\begin{proposition}\label{CoercivityProp} Let $\phi^\varepsilon_i: (\mathbb{R}^n)^{Z_\varepsilon(\Omega_i)} \to [0,+\infty)$ satisfy {\rm(H3)}. If $u \in L^p(\Omega,\mathbb{R}^n)$ is such that $F'(u,A) < +\infty$, then $u \in W^{1,p}(A,\mathbb{R}^n)$ and
\begin{equation}\label{lowerbound}
F'(u,A) \geq c\left(||\nabla u||^p_{L^p(A;\mathbb{R}^{n\times N})}- |A|\right)
\end{equation}
for some positive constant $c$ independent on $u$ and $A$.
\end{proposition}
\begin{proof}
Let $\varepsilon_n \to 0^+$ and let $u_n \to u $ in $L^p(\Omega;\mathbb{R}^n) $ be such that $\underset{n}{\liminf} F_{\varepsilon_n}(u_n,A) <+\infty$. By (H3) we get
\begin{align}\label{coercivitybound}
F_{\varepsilon_n}(u_n,A) \geq c \sum_{i \in Z_\varepsilon(A)}\sum_{k=1}^N \varepsilon^N |D^{e_k}_{\varepsilon_n}u_n(i)|^p - c N |A|.
\end{align}
For any $k \in \{1,\cdots,N\}$, consider the sequence of piecewise-affine functions $(v_n^k)$ defined as follows
\begin{align*}
v_n^k(x) =u_n(i) + D^{e_k}_{\varepsilon_n}u_n(i)(x_k-i_k) \quad x \in (i + [0,\varepsilon_n)^N) \cap \Omega,\  i \in Z_\varepsilon(A).
\end{align*}
Note that $v_n^k$ is a function of bounded variation and we will denote by $\frac{\partial v^k_n}{\partial x_k}$ the density of the absolutely continuos part of $D_{x_k}v^k_n$ with respect to the Lebesgue measure. Moreover, for $\mathcal{H}^{N-1}$-a.e.~$y \in (A)^{e_k}$ the slices $(v_n^k)_{e_k,y}$ belong to $W^{1,p}((A)^{e_k}_y;\mathbb{R}^n)$. 
Note that, for any fixed $\eta >0 $, $v_n^k \to u $ in $L^p(A_\eta;\mathbb{R}^n)$ for every $k \in \{1,\cdots,N\}$. Moreover, since $\frac{\partial v^k_n}{\partial x_k}(x) = D^{e_k}_{\varepsilon_n}u_n(i)$ for $x \in i +[0,\varepsilon_n)^N$, we get
\begin{align*}
F_{\varepsilon_n}(u_n,A) \geq c \sum_{k=1}^N \int_{A_\eta} \left|\frac{\partial v^k_n}{\partial x_k}(x)\right|^p \mathrm{d}x - c N|A|.
\end{align*}
We now apply a standard slicing argument. By Fubini's Theorem and Fatou's Lemma for any $k$ we get
\begin{align*}
\liminf_n \int_{A_\eta}\left|\frac{\partial v^k_n}{\partial x_k}(x)\right|^p \mathrm{d}x \geq \int_{(A_\eta)^{e_k}} \liminf_n \int_{(A)^{e_k}_y} |(v^k_n)'_{e_k,y}(t)|^p\mathrm{dt} \mathrm{d}\mathcal{H}^{N-1}(y).
\end{align*}
Since, up to passing to a subsequence, we may assume that, for $\mathcal{H}^{N-1}$-a.e.~$y \in (A_\eta)^{e_k}$ $(v^k_n)_{e_k,y} \to u_{e_k,y}$ in $L^p((A_\eta)^{e_k}_y;\mathbb{R}^n)$, we deduce that $u_{e_k,y} \in W^{1,p}((A_\eta)_y^{e_k};\mathbb{R}^n) $ for $\mathcal{H}^{N-1}$-a.e.~$y \in (A_\eta)^{e_k}$ and
\begin{align*}
\liminf_n \int_{A_\eta}\left|\frac{\partial v^k_n}{\partial x_k}(x)\right|^p \mathrm{d}x \geq \int_{(A_\eta)^{e_k}} \int_{(A)^{e_k}_y} |u'_{e_k,y}(t)|^p\mathrm{dt} \mathrm{d}\mathcal{H}^{N-1}(y).
\end{align*}
Then by (\ref{coercivitybound}), we have
\begin{align*}
\liminf_n F_{\varepsilon_n}(u_n,A) \geq  c  \sum_{k=1}^N\int_{(A_\eta)^{e_k}} \int_{(A)^{e_k}_y} |u'_{e_k,y}(t)|^p\mathrm{dt} \mathrm{d}\mathcal{H}^{N-1}(y) - c N |A|.
\end{align*}
Since, in particular, the previous inequality implies that
 \begin{align*}
 \sum_{k=1}^N\int_{(A_\eta)^{e_k}} \int_{(A)^{e_k}_y} |u'_{e_k,y}(t)|^p\mathrm{dt} \mathrm{d}\mathcal{H}^{N-1}(y) < +\infty,
 \end{align*}
 thanks to the characterization of $W^{1,p}$ by slicing we obtain that $u \in W^{1,p}(A_\eta,\mathbb{R}^n)$ and 
 \begin{align*}
 \liminf_n F_{\varepsilon_n}(u_n,A) &\geq c\sum_{k=1}^N\int_{A_\eta}  \left|\frac{\partial u}{\partial x_k}(x)\right|^p \mathrm{d}x - cN |A| \\&\geq c\left(||\nabla u||^p_{L^p(A_\eta;\mathbb{R}^{n\times N})}- |A|\right)
 \end{align*}
 Letting $\eta \to 0^+$, we get the conclusion.
\end{proof}

\begin{proposition} \label{CoercivityProp} Let $\phi^\varepsilon_i: (\mathbb{R}^n)^{Z_\varepsilon(\Omega_i)} \to [0,+\infty)$ satisfy {\rm (H2),(H4)} and {\rm(H5)}. We then have
\begin{equation}
F''(u,A) \leq C\left(||\nabla u||^p_{L^p(A;\mathbb{R}^{n\times N})}+ |A|\right)
\end{equation}
for some positive constant $C$ independent on $u$ and $A$.
\end{proposition}
\begin{proof} We first show that the inequality holds for $u\in W^{1,p}(\Omega;\mathbb{R}^n)$ piecewise affine and then we recover the inequality for any $u \in W^{1,p}(\Omega;\mathbb{R}^n)$ by a density argument. Let $u\in W^{1,p}(\Omega;\mathbb{R}^n)$ be piecewise affine, that means 
\begin{align*}
u(x) = \sum_{k=1}^K \chi_{\Omega_k}(x) (M_k x + b_k)= \sum_{k=1}^K \chi_{\Omega_k}(x)u_k(x),
\end{align*} 
where $\Omega_k= U_k \cap \Omega$, with $U_k$ disjoint open simplices such that $|\Omega \setminus \bigcup_k \Omega_k|=0$, $M_k\in \mathbb{R}^{n\times N}, b_k \in \mathbb{R}^n, k=1,\ldots,K$. In the following, for such $u\in W^{1,p}(\Omega;\mathbb{R}^n)$  we construct $u_\delta \in W^{1,\infty}(\Omega;\mathbb{R}^n)$, which agrees with $u$ on $\Omega^\delta_k$ for all $k \in \{1,\ldots,K\}$, $u_\delta=u$ on $(\Omega_k)^\delta$ for all $k \in \{1,\ldots,K\}$ and close to $\partial \Omega_k$ we have that $u_\delta= \psi_k^\delta u_{j} + (1-\psi_k^\delta) u^\delta_{j-1}$ for some $j \in \{1,\ldots,K\}$ and $||\nabla u_j ||_\infty + ||\nabla u^\delta_{j-1} ||_\infty \leq C||\nabla u||_\infty$ independent on $\delta$. The way we construct $u^\delta_{j}$, it satisfies the same property close to the boundary so that (H5) or (H4) can be applied repeatedly. In $\bigcup_k \Omega_k^{2\delta}$ we estimate the interactions separately with (H4). 

 Let 
$$\mathrm{d}= \min_{\partial \Omega_k \cap \Omega_j = \emptyset}\mathrm{dist}_\infty(\partial \Omega_k, \partial \Omega_j)$$
and let $\delta < \frac{\mathrm{d}}{4}$. For $k=1,\ldots, K$ choose inductively $ \varphi_{j,\delta}^k \subset C^\infty(\Omega), \, j =1,\ldots, k$ such that
\begin{align*}
&0\leq \varphi^k_{j,\delta} \leq 1, \quad \sum_{j =1}^k \varphi^k_{j,\delta} =1, \quad \mathrm{supp}(\varphi^k_{j,\delta}) \subset (\Omega_{j})_\delta,\quad \varphi^k_{j,\delta} = (1-\varphi^k_{k,\delta}) \varphi^{k-1}_{j,\delta} \text{ for all } j < k \\& \varphi_{j,\delta}^k(x) =1 \text{ if } x \in (\Omega_{j})^\delta, \quad ||\nabla \varphi_{j,\delta}^k||_\infty \leq \frac{C}{\delta}
\end{align*}
and define
\begin{align*}
u^k_\delta(x)= \varphi^k_{k,\delta}(x)u_k(x) + (1-\varphi_{k,\delta}^k(x)) u^{k-1}_\delta(x), \quad u^1_\delta(x) = u_1(x).
\end{align*}
Set $u_\delta(x)=u^K_\delta(x)$. We then have $||\nabla u_\delta^k||_\infty \leq C||\nabla u||_\infty$ for all $k \in \{1,\ldots,K\}$ and $u_\delta \to u$ strongly in $W^{1,p}(\Omega;\mathbb{R}^n)$ and we claim that
\begin{align}\label{UpperBound}
\liminf_{\delta \to 0}F''(u_\delta,A) \leq C(|| \nabla u ||_{L^p(A;\mathbb{R}^{d \times N})}^p + |A|).
\end{align} 
To this end define
\begin{align*}
u_\delta^\varepsilon(i) = u_\delta(i), \quad i \in Z_\varepsilon(\Omega).
\end{align*}
We have that $u^\varepsilon_\delta \to u$ strongly in $L^p(\Omega;\mathbb{R}^n)$ and therefore
\begin{align}\label{lsc udelta}
F''(u_\delta,A) \leq \liminf_{\varepsilon \to 0} F_\varepsilon(u_\delta^\varepsilon,A).
\end{align}
We divide the energy into the energy of points which are far away from the boundary of all the $\Omega_k$ and to the points which are close to some of the boundaries of $\Omega_k$:
\begin{align*}
F_\varepsilon(u_\delta^\varepsilon,A) = \sum_{i \in Z_\varepsilon(A)} \varepsilon^N \phi_i^\varepsilon(\{(u^\varepsilon_\delta)_{j+i}\}_{j \in Z_\varepsilon(\Omega_i)}) &= \sum_{k=1}^K \sum_{i \in Z_\varepsilon(A \cap (\Omega_k)^{2\delta})} \varepsilon^N \phi_i^\varepsilon(\{(u^\varepsilon_\delta)_{j+i}\}_{j \in Z_\varepsilon(\Omega_i)}) \\&+ \sum_{i \in Z_\varepsilon\left(A \setminus \left(\bigcup_{k=1}^K (\Omega_k)^{2\delta}\right) \right)} \varepsilon^N \phi_i^\varepsilon(\{(u^\varepsilon_\delta)_{j+i}\}_{j \in Z_\varepsilon(\Omega_i)})\\&=I_{\varepsilon,\delta}^1 +I_{\varepsilon,\delta}^2.
\end{align*}
Now, note that for $M \in \mathbb{R}^{n\times N}$, $ z \in \mathbb{R}^n$ and $(Mx + z)(i)= Mi +z$  we have
\begin{align*}
|D^{e_n}_\varepsilon (Mx+z)(i)|= \left|\frac{M(i+\varepsilon e_n)+z- (Mi+z)}{\varepsilon}\right| \leq |M|, \quad \forall n \in \{1,\ldots,N\}.
\end{align*}
Moreover, note that for every $v \in W^{1,\infty}(\Omega;\mathbb{R}^n)$ we have that
$
|D^\xi_\varepsilon v|^p \leq ||\nabla v||_\infty^p.
$
Using (H4) and (H2), noting that $\nabla u(x) = \nabla u_k(x)= M_k$ for $x \in \Omega_k$ and using the fact that $u^\varepsilon_\delta(j)=u_k(j)$ for all $j \in Z_\varepsilon(Q_\delta(i)), i \in Z_\varepsilon((\Omega_k)^{2\delta})$  (with slight abuse of notation we write $u_k$ for the discrete function as well as for the function defined in the continuum)  we can estimate the first term by
\begin{align*}
I_{\varepsilon,\delta}^1&\leq \sum_{k=1}^K \sum_{i \in Z_\varepsilon(A \cap (\Omega_k)^{2\delta})} \varepsilon^N\phi_i^\varepsilon(\{(u_k)_{j+i}\}_{j \in Z_\varepsilon(\Omega_i)}) \\&+\sum_{k=1}^K \sum_{i \in Z_\varepsilon(A \cap (\Omega_k)^{2\delta})} \varepsilon^N\underset{j+\varepsilon\xi \in Z_\varepsilon(\Omega)}{\sum_{j \in Z_\varepsilon(\Omega),\xi \in \mathbb{Z}^N}}C^{j-i,\xi}_{\varepsilon,\delta} (|D^\xi_\varepsilon u^\varepsilon_\delta(j)|^p+1)  \\&\leq \sum_{k=1}^K \sum_{i \in Z_\varepsilon(A \cap \Omega_k)} \varepsilon^N C(|M_k|^p+1) + C(u)  \sum_{j \in Z_\varepsilon(\Omega),\xi \in \mathbb{Z}^N}C^{j-i,\xi}_{\varepsilon,\delta} \\& \leq C(|| \nabla u||_{L^p(A_\varepsilon;\mathbb{R}^{n\times N})} + |A_\varepsilon|) + C(u)  \sum_{j \in Z_\varepsilon(\Omega),\xi \in \mathbb{Z}^N}C^{j-i,\xi}_{\varepsilon,\delta}.
\end{align*}
Taking the $\limsup$ as $\varepsilon \to 0$ taking into account (\ref{Assumptions on Cj}) and using the dominated-convergence theorem we obtain
\begin{align}\label{Bulkbound}
\limsup_{\varepsilon \to 0}I_{\varepsilon,\delta}^1 \leq C(|| \nabla u||_{L^p(A;\mathbb{R}^{n\times N})} + |A|).
\end{align}
Now let $i \in Z_\varepsilon(A \setminus (\bigcup_{k=1}^K(\Omega_k)^{2\delta}))$, that means $\mathrm{dist}_\infty(i,\partial \Omega_k)\leq 2\delta$ for some $k \in \{1,\ldots,K\}$. We prove
\begin{align}\label{BoundPhi}
\phi_i^\varepsilon(\{(u^\varepsilon_\delta)_{j+i}\}_{j \in Z_\varepsilon(\Omega_i)}) \leq C(u) 
\end{align}
for some constant depending on $u$. Recall that $u_\delta = u^K_\delta$, take $k \in \{1,\ldots,K\}$ and assume that we have proved already that
\begin{align*}
\phi_i^\varepsilon(\{(u^\varepsilon_\delta)_{j+i}\}_{j \in Z_\varepsilon(\Omega_i)}) \leq C(u)\left(\phi_i^\varepsilon(\{((u^k_\delta)^\varepsilon)_{j+i}\}_{j \in Z_\varepsilon(\Omega_i)})+1\right);
\end{align*}
we then prove that
\begin{align} \label{inductionbound}
\phi_i^\varepsilon(\{(u^\varepsilon_\delta)_{j+i}\}_{j \in Z_\varepsilon(\Omega_i)}) \leq C(u)\left(\phi_i^\varepsilon(\{((u^{k-1}_\delta)^\varepsilon)_{j+i}\}_{j \in Z_\varepsilon(\Omega_i)})+1\right).
\end{align}
We either have $\varphi^k_{k,\delta}=0$ in $Q_\delta(i)$. Then by using $||\nabla u^{k-1}_\delta ||_\infty \leq C||\nabla u||_\infty$, $|D^\xi_\varepsilon u| \leq C ||\nabla u||_\infty $ and (\ref{Assumptions on Cj}) we obtain
\begin{align*}
\phi_i^\varepsilon(\{((u^k_\delta)^\varepsilon)_{j+i}\}_{j \in Z_\varepsilon(\Omega_i)}) &\leq\phi_i^\varepsilon(\{((u^{k-1}_\delta)^\varepsilon)_{j+i}\}_{j \in Z_\varepsilon(\Omega_i)})+  \underset{j+\varepsilon\xi \in Z_\varepsilon(\Omega)}{\sum_{j \in Z_\varepsilon(\Omega),\xi \in \mathbb{Z}^N}} C^{j-i,\xi}_{\varepsilon,\delta}\Big(|D^\xi_\varepsilon(u^k_\delta)^\varepsilon(j)|^p +1\Big) \\&\leq C(u)(\phi_i^\varepsilon(\{((u^{k-1}_\delta)^\varepsilon)_{j+i}\}_{j \in Z_\varepsilon(\Omega_i)}) + 1)
\end{align*}
and we obtain (\ref{inductionbound}). Now in the case that $\varphi^k_{k,\delta}(x) >0$ for some $x \in Q_\delta(j)$ we use (H5) with $\varphi^k_{k,\delta}$ as a cutoff function, $u_k,u^{k-1}_\delta $ as $z,w$  and the assumptions on $\varphi_{k,\delta}^k$ we obtain
\begin{align*}
\phi_i^\varepsilon(\{((u^{k}_\delta)^\varepsilon)_{j+i}\}_{j \in Z_\varepsilon(\Omega_i)})\leq &C\left(\phi_i^\varepsilon(\{(u_k^\varepsilon)_{j+i}\}_{j \in Z_\varepsilon(\Omega_i)})+\phi_i^\varepsilon(\{((u^{k-1}_\delta)^\varepsilon)_{j+i}\}_{j \in Z_\varepsilon(\Omega_i)})+1\right)\\& +R^\varepsilon_i(u_k,u^{k-1}_\delta,\varphi^k_{k,\delta}) 
\end{align*}
with
\begin{align*}
R^\varepsilon_i(u_k,u^{k-1}_\delta,\varphi^k_{k,\delta}) =\left(\frac{1}{\delta^p}+1\right) &\underset{j+\varepsilon\xi \in Z_\varepsilon(\Omega)}{\sum_{j \in Z_\varepsilon(\Omega),\xi \in \mathbb{Z}^N}}C^{j-i,\xi}_{\varepsilon} |(u_k)(j+\varepsilon\xi)- (u^{k-1}_\delta)(j+\varepsilon\xi)|^p \\+ &\underset{j+\varepsilon\xi \in Z_\varepsilon(\Omega)}{\sum_{j \in Z_\varepsilon(\Omega),\xi \in \mathbb{Z}^N}} C_{\varepsilon}^{j-i,\xi} \left(|D^\xi_\varepsilon(u_\delta^{k-1})^\varepsilon(j)|^p + |D^\xi_\varepsilon u_{k}^\varepsilon(j)|^p  \right).
\end{align*}
First, note that by (H2) and by $ ||\nabla u_k||_\infty \leq C ||\nabla u||_\infty $ we have
\begin{align*}
\phi_i^\varepsilon(\{(u_k^\varepsilon)_{j+i}\}_{j \in Z_\varepsilon(\Omega_i)}) \leq C(|M_k|^p+1) \leq C(u)
\end{align*}
and 
\begin{align}\label{Gradientbound}
\underset{j+\varepsilon\xi \in Z_\varepsilon(\Omega)}{\sum_{j \in Z_\varepsilon(\Omega),\xi \in \mathbb{Z}^N}} C_{\varepsilon}^{j-i,\xi} \left(|D^\xi_\varepsilon(u_\delta^{k-1})^\varepsilon(j)|^p + |D^\xi_\varepsilon u_{k}^\varepsilon(j)|^p  \right)\leq C(u)
\end{align}
since  $||\nabla u_k ||_\infty, ||\nabla u^{k-1}_\delta||_\infty \leq C||\nabla u||_\infty$.
Now since $\varphi^k_{k,\delta}(x) >0$ for some $x \in Q_\delta(i)$  we have that $\mathrm{dist}_\infty(i,\partial \Omega_k) < 2\delta$, $u_k(x) = u^{k-1}_\delta(x) $ on $\partial \Omega_k$ and $||\nabla u_k ||_\infty, ||\nabla u^{k-1}_\delta||_\infty \leq C||\nabla u||_\infty$. We therefore have
\begin{align*}
|(u_k)(j) - (u^{k-1}_\delta)(j)| &\leq |(u_k)(j) - (u_k)(x)| + |(u_k)(x) - (u^{k-1}_\delta)(x)|  + | (u^{k-1}_\delta)(x) - (u^{k-1}_\delta)(j)|\\& \leq C(u)\delta
\end{align*}
for all $j \in Z_\varepsilon(Q_{2\delta}(i))$ and hence we have, splitting the sum into the summation over $j,\xi$ such that $\max\{\varepsilon|\xi|,|j-i|\} > \delta\} $ and the complement and using (\ref{Gradientbound}) we obtain
\begin{align*}
 R^\varepsilon_i(u^{k-1}_\delta,u_k,\varphi^k_{k,\delta})&\leq C(u)\Big((1+\delta^p) \underset{\max\{\varepsilon|\xi|,|j-i|\}\leq \delta}{\underset{j+\varepsilon\xi \in Z_\varepsilon(\Omega)}{\sum_{j \in Z_\varepsilon(\Omega),\xi \in \mathbb{Z}^N}}} C^{j-i,\xi}_{\varepsilon}+ \Bigl(\frac{1}{\delta^p}+1\Bigr)\underset{\max\{\varepsilon|\xi|,|j-i|\}>\delta}{\underset{j+\varepsilon\xi \in Z_\varepsilon(\Omega)}{\sum_{j \in Z_\varepsilon(\Omega),\xi \in \mathbb{Z}^N}}}C^{j-i,\xi}_\varepsilon +1\Big)  \\&\leq C(u).
\end{align*}
for $\varepsilon$ small enough, using (\ref{Assumptions on Cj}). By summing over $j$ and, taking the maximum over $j \in Z_\varepsilon(\Omega) $ in the inner sum and using (\ref{Assumptions on Cj}) we obtain (\ref{inductionbound}).
If $k=1$ by (H2) and the definition of $u^1_\delta$ we have that
\begin{align*}
\phi_i^\varepsilon(\{((u^{1}_\delta)^\varepsilon)_{j+i}\}_{j \in Z_\varepsilon(\Omega_i)}) = \phi_i^\varepsilon(\{(u_1)_{j+i}\}_{j \in Z_\varepsilon(\Omega_i)}) \leq C(u)
\end{align*}
and (\ref{BoundPhi}) follows. Now for $A \in \mathcal{A}(\Omega)$ we have that
\begin{align*}
I_{\varepsilon,\delta}^2 \leq C(u) \varepsilon^N\#  Z_\varepsilon\left(A\setminus(\bigcup_{k=1}^K(\Omega_k)^{2\delta})\right) &\leq C(u) \varepsilon^N\#  Z_\varepsilon\left(\Omega\setminus(\bigcup_{k=1}^K(\Omega_k)^{2\delta})\right) \\&\leq C(u) \left| \Omega\setminus(\bigcup_{k=1}^K(\Omega_k)^{2\delta})\right|.
\end{align*}
Therefore, using that $|\Omega \setminus \bigcup_{k=1}^K \Omega_k| = 0$ and the dominated-convergence theorem, we have that
\begin{align} \label{surfacezero}
\limsup_{\delta \to 0}\limsup_{\varepsilon \to 0} I^2_{\varepsilon,\delta} = 0
\end{align}
By (\ref{lsc udelta}), (\ref{Bulkbound}) and (\ref{surfacezero}) we obtain (\ref{UpperBound}) and the claim follows.
Now by the lower semicontinuity of $F''(\cdot,A)$ we have
\begin{align}\label{Fdoubleprimeu}
F''(u,A) \leq \liminf_{\delta \to 0} F''(u_\delta,A) &\leq C(|| \nabla u ||_{L^p(A;\mathbb{R}^{n\times N})}^p + |A|).
\end{align}
Now for general $u \in W^{1,p}(\Omega;\mathbb{R}^n)$ we take $\{u_n\} \subset W^{1,p}(\Omega;\mathbb{R}^n)$ piecewise affine such that $u_n \to u$ strongly in $W^{1,p}(\Omega;\mathbb{R}^n)$ and again by the lower semicontinuity of $F''(\cdot,A)$ we have
\begin{align*}
F''(u,A) \leq \liminf_{n \to \infty} F''(u_n,A) \leq \lim_{n \to \infty} C(|| \nabla u_n ||_{L^p(A;\mathbb{R}^{d \times N})}^p + |A|) = C(|| \nabla u ||_{L^p(A;\mathbb{R}^{n\times N})}^p + |A|)
\end{align*}
and the statement is proven.
\end{proof}

\begin{proposition}\label{Subad} Let $\phi_i^\varepsilon : (\mathbb{R}^n)^{Z_\varepsilon(\Omega)} \to [0,+\infty)$ satisfy {\rm (H2)--(H5)}. Let $A,B \in \mathcal{A}(\Omega)$ and let $A',B' \in \mathcal{A}(\Omega)$ be such that $A' \subset\subset A$ and $B' \subset\subset B$. Then for any $u \in W^{1,p}(\Omega;\mathbb{R}^n)$ we have
\begin{equation*}
F''(u, A' \cup B' ) \leq F''(u,A) + F''(u,B)
\end{equation*}
\end{proposition}
\begin{proof}
Without loss of generality, we may suppose $F''(u,A)$ and $F''(u,B)$ finite. Let $(u_\varepsilon)_\varepsilon$ and $(v_\varepsilon)_\varepsilon$ converge to $u $ in $L^p(\Omega;\mathbb{R}^n)$ and be such that
\begin{equation*}
\limsup_{\varepsilon \to 0^+} F_\varepsilon(u_\varepsilon,A) = F''(u,A), \quad \limsup_{\varepsilon \to 0^+} F_\varepsilon(v_\varepsilon,B) = F''(u,B),
\end{equation*}
and therefore
\begin{align}\label{Energybound1}
&\sup_{\varepsilon >0} \sum_{i \in Z_\varepsilon(A)} \varepsilon^N \phi_i^\varepsilon(\{(u_\varepsilon)_{j+i}\}_{j \in Z_\varepsilon(\Omega_i)}) < \infty, \\&\label{Energybound2}\sup_{\varepsilon >0} \sum_{i \in Z_\varepsilon(B)} \varepsilon^N \phi_i^\varepsilon(\{(v_\varepsilon)_{j+i}\}_{j \in Z_\varepsilon(\Omega_i)}) < \infty.
\end{align}
By (H3) we have that
\begin{align} \label{Coercivitybounds1}
&\sup_{n\in \{1,\ldots,N\}} \sup_{\varepsilon >0} \sum_{i \in Z_\varepsilon(A'')} \varepsilon^N |D^{e_n}_\varepsilon u_\varepsilon(i)|^p < +\infty \\&\label{Coercivitybounds2} \sup_{n\in \{1,\ldots,N\}} \sup_{\varepsilon >0} \sum_{i \in Z_\varepsilon(B'')} \varepsilon^N |D^{e_n}_\varepsilon v_\varepsilon(i)|^p < +\infty
\end{align}
for all $A'' \subset\subset A, B'' \subset\subset B$.
Since $u_\varepsilon$ and $v_\varepsilon$ converge to $u$ in $L^p(\Omega;\mathbb{R}^n)$, we have that
\begin{align} \label{convergencebounds1}
\sum_{i \in Z_\varepsilon(\Omega)}\varepsilon^N \left(|u_\varepsilon(i)|^p + |v_\varepsilon(i)|^p \right) & \leq ||u_\varepsilon||^p_{L^p(\Omega;\mathbb{R}^n)} +||v_\varepsilon||^p_{L^p(\Omega;\mathbb{R}^n)} \leq C <\infty \\ \label{convergencebounds2} \sum_{i \in Z_\varepsilon(\Omega)}\varepsilon^N \left(|u_\varepsilon(i)-v_\varepsilon(i)|^p \right) &\leq ||u_\varepsilon-v_\varepsilon||_{L^p(\Omega;\mathbb{R}^n)} \to 0.
\end{align}
Since $u \in W^{1,p}(\Omega;\mathbb{R}^n)$ there exists $\tilde{u}_\varepsilon, \tilde{v}_\varepsilon$ such that $\tilde{u}_\varepsilon$ and $\tilde{v}_\varepsilon$ converge to $u$ in $L^p(\Omega;\mathbb{R}^n)$ and 
\begin{align}\label{smoothapproxbound}
\sup_{n \in \{1,\ldots,N\}} \sup_{\varepsilon >0} \sum_{i \in Z_\varepsilon(\Omega)} \varepsilon^N\left( |D_\varepsilon^{e_n} \tilde{u}_\varepsilon(i)|^p + |D_\varepsilon^{e_n} \tilde{v}_\varepsilon(i)|^p\right) < \infty.
\end{align}
Take $A'',A''',B'',B''' \in \mathcal{A}(\Omega)$,$\varphi_A , \varphi_B \in C^\infty(\Omega)$  such that $A' \subset\subset A'' \subset \subset A''' \subset\subset A$, $B' \subset\subset B'' \subset \subset B''' \subset\subset B$, $0\leq \varphi_A,\varphi_B\leq 1$, $A'''\subset\{\varphi_A=0\} $, $B'''\subset\{\varphi_B=0\} $, $A''\subset\{\varphi_A=1\}$, $B'' \subset\{\varphi_B=1\}  $ and $||\nabla \varphi_A||_\infty, ||\nabla \varphi_B||_\infty \leq C$, and define $u'_\varepsilon = \varphi_A u_\varepsilon + (1-\varphi_A) \tilde{u}_\varepsilon$,$v'_\varepsilon = \varphi_B v_\varepsilon + (1-\varphi_B) \tilde{v}_\varepsilon$.
Now for $j \in Z_\varepsilon(\Omega)$, $\psi $ cut-off function $z,w \in \mathcal{A}_\varepsilon(\Omega;\mathbb{R}^n)$ $v = \psi z + (1-\psi)w$ we have
\begin{align}\label{convexcomb}
D^{e_n}_\varepsilon v(j) = \psi(j)D^{e_n}_\varepsilon z(j) + (1-\psi(j))D^{e_n}_\varepsilon w(j) + D^{e_n}_\varepsilon \psi(j)(z(j)-w(j))
\end{align}
Since $\{\varphi_A >0\} \subset\subset A$, by (\ref{Coercivitybounds1}), (\ref{smoothapproxbound}) and (\ref{convexcomb})  we have that
\begin{align} \label{smoothapproxbound2}
\sup_{n\in \{1,\ldots,N\}}\sup_{\varepsilon >0}\sum_{j \in Z_\varepsilon(\Omega)} \varepsilon^N |D^{e_n}_\varepsilon u'_\varepsilon(j)|^p <\infty.
\end{align}
We can perform a similar construction for $v'_\varepsilon$ and therefore assume that an analogous bound to (\ref{smoothapproxbound2}) holds also for $v'_\varepsilon$. Moreover, since $u'_\varepsilon$ and $v'_\varepsilon $ converge to $u$ in $L^p(\Omega;\mathbb{R}^n)$ we have that (\ref{convergencebounds1}) and (\ref{convergencebounds2}) hold with $u'_\varepsilon$ and $v'_\varepsilon$.
 Now for $\delta >0 $, by (H4), it holds
\begin{align} \label{Fepsprime}
\phi_i^\varepsilon(\{(u_\varepsilon')_{j+i}\}_{j \in Z_\varepsilon(\Omega_i)}) \leq \phi_i^\varepsilon(\{(u_\varepsilon)_{j+i}\}_{j \in Z_\varepsilon(\Omega_i)}) +\underset{j+\varepsilon\xi \in Z_\varepsilon(\Omega)}{\sum_{j \in Z_\varepsilon(\Omega),\xi \in \mathbb{Z}^N}} C^{j-i,\xi}_{\varepsilon,\delta} (|D^\xi_\varepsilon u'_\varepsilon(j)|^p+1)
\end{align}
as well as a similar estimate for $v'_\varepsilon$ in $B'$.
Set 
\begin{align*}
\mathrm{d}:= \mathrm{dist}_\infty(A',A^c)\qquad\hbox{and}\qquad A_k := (A')_{ \frac{k}{3K}\mathrm{d}}
\end{align*}
for any $k \in \{K, \ldots,2K\}$. Let $\varphi_k$ be a cut-off function between $A_k$ and $A_{k+1}$, with $||\nabla \varphi_k||_\infty \leq CK$ . Then for any $k \in \{K,\ldots, 2K \}$ consider the family of functions $w_\varepsilon^k\in \mathcal{A}_\varepsilon(\Omega;\mathbb{R}^n)$ converging to $u$ in $L^p(\Omega;\mathbb{R}^n) $, defined as
\begin{align*}
w_\varepsilon^k(i) = \varphi_k(i) u'_\varepsilon(i) + (1-\varphi_k(i)) v'_\varepsilon(i).
\end{align*}
Given $i \in Z_\varepsilon(A' \cup B')$, then either
$\mathrm{dist}_\infty(i, A_{k+1} \setminus \overline{A}_k) \geq \frac{\mathrm{d}}{3K}$, in which case either $w^k_\varepsilon(j) = u'_\varepsilon(j)$ for $j \in Z_\varepsilon(Q_{\frac{\mathrm{d}}{2K}}(i))$ and $i \in Z_\varepsilon(A_k)$ or $w^k_\varepsilon(j) = v'_\varepsilon(j)$ $j \in Z_\varepsilon(Q_{\frac{\mathrm{d}}{2K}}(i))$ and $i \in Z_\varepsilon((A' \cup B') \setminus A_{k+1}) \subset Z_\varepsilon(B')$, or $\mathrm{dist}_\infty(i, A_{k+1} \setminus \overline{A}_k) < \frac{\mathrm{d}}{6K}$ . In the first case, using (H4), we estimate
\begin{align} \label{Awayfrom1}
\phi_i^\varepsilon(\{(w^k_\varepsilon)_{j+i}\}_{j \in Z_\varepsilon(\Omega)}) 
\leq \phi_i^\varepsilon(\{(u'_\varepsilon)_{j+i}\}_{j \in Z_\varepsilon(\Omega)}) +\underset{j+\varepsilon\xi \in Z_\varepsilon(\Omega)}{\sum_{j \in Z_\varepsilon(\Omega),\xi \in \mathbb{Z}^N}} C^{j-i,\xi}_{\varepsilon,\frac{\mathrm{d}}{2K}} (|D^\xi_\varepsilon w^k_\varepsilon(j)|^p+1).
\end{align}
In the second case, using (H4), we estimate
\begin{align}\label{Awayfrom2}
\phi_i^\varepsilon(\{(w^k_\varepsilon)_{j+i}\}_{j \in Z_\varepsilon(\Omega)}) 
\leq \phi_i^\varepsilon(\{(v'_\varepsilon)_{j+i}\}_{j \in Z_\varepsilon(\Omega)}) +\underset{j+\varepsilon\xi \in Z_\varepsilon(\Omega)}{\sum_{j \in Z_\varepsilon(\Omega),\xi \in \mathbb{Z}^N}} C^{j-i,\xi}_{\varepsilon,\frac{\mathrm{d}}{2K}} (|D^\xi_\varepsilon w^k_\varepsilon(j)|^p+1).
\end{align}
Using (\ref{convexcomb}) and the convexity of $|\cdot|^p$ we have for $j \in Z_\varepsilon(\Omega)$ and $\xi \in \mathbb{Z}^N$
\begin{align} \label{Boundawayfrom}
 |D^{\xi}_\varepsilon w^k_\varepsilon(j)|^p\leq &   |D^{\xi}_\varepsilon u'_\varepsilon(j)|^p + |D^{\xi}_\varepsilon v'_\varepsilon(j)|^p+CK^p |u'_\varepsilon(j+\varepsilon\xi)-v'_\varepsilon(j+\varepsilon\xi)|^p.
\end{align}
Now if $\mathrm{dist}_\infty(i, A_{k+1} \setminus \overline{A}_k) < \frac{\mathrm{d}}{3K}$ we have that $i \in Z_\varepsilon(A_{k+2}\setminus \overline{A}_{k-1}) =: Z_\varepsilon(S_k)$ where $S_k \subset\subset A \cap B$. By (H5) we have that for such an $i $ it holds
\begin{align} \label{Convexityclose}
\phi_i^\varepsilon(\{(w^k_\varepsilon)_{j+i}\}_{j \in Z_\varepsilon(\Omega)}) \leq &C(\phi_i^\varepsilon(\{(v'_\varepsilon)_{j+i}\}_{j \in Z_\varepsilon(\Omega)}) +\phi_i^\varepsilon(\{(u'_\varepsilon)_{j+i}\}_{j \in Z_\varepsilon(\Omega)}))+ R^\varepsilon_i(u'_\varepsilon,v'_\varepsilon,\varphi_k)
\end{align}
where
\begin{align}\label{Rieps}
R_i^\varepsilon(u'_\varepsilon,v'_\varepsilon,\varphi_k) =(CK^p+1)&\underset{j+\varepsilon\xi \in Z_\varepsilon(\Omega)}{\sum_{j \in Z_\varepsilon(\Omega),\xi \in \mathbb{Z}^N}} C^{j-i,\xi}_\varepsilon |u_\varepsilon(j+\varepsilon\xi)-v_\varepsilon(j+\varepsilon\xi)|^p \\+\nonumber  &\underset{j+\varepsilon\xi \in Z_\varepsilon(\Omega)}{\sum_{j \in Z_\varepsilon(\Omega),\xi \in \mathbb{Z}^N}} C^{j-i,\xi}_\varepsilon \left(|D^\xi_\varepsilon u'_\varepsilon(j)|^p+|D^\xi_\varepsilon v'_\varepsilon(j)|^p +1\right).
\end{align} 
Summing over $i \in Z_\varepsilon(A'\cup B')$ and splitting into the two cases as described above, using (\ref{Awayfrom1})--(\ref{Rieps}), we have
\begin{align*}
F_\varepsilon(w^k_\varepsilon,A'\cup B') &\leq  \underset{\mathrm{dist}_\infty(i, A_{k+1} \setminus \overline{A}_k) \geq \frac{\mathrm{d}}{3K}}{\sum_{i \in Z_\varepsilon(A' \cup B')}} \varepsilon^N \phi_i^\varepsilon(\{(w^k_\varepsilon)_{j+i}\}_{j \in Z_\varepsilon(\Omega_i)})+\sum_{i \in Z_\varepsilon(S_k)} \varepsilon^N \phi_i^\varepsilon(\{(w^k_\varepsilon)_{j+i}\}_{j \in Z_\varepsilon(\Omega_i)})  \\ \leq &F_\varepsilon(u_\varepsilon,A) + F_\varepsilon(v_\varepsilon,B) \\&+CK^p\sum_{i \in Z_\varepsilon(S_k)}\varepsilon^N\underset{j+\varepsilon\xi \in Z_\varepsilon(\Omega)}{\sum_{j \in Z_\varepsilon(\Omega),\xi \in \mathbb{Z}^N}} C^{j-i,\xi}_\varepsilon |u'_\varepsilon(j+\varepsilon\xi)-v'_\varepsilon(j+\varepsilon\xi)|^p  \\&+  CK^p\sum_{i \in Z_\varepsilon(A'\cup B')}\varepsilon^N \underset{j+\varepsilon\xi \in Z_\varepsilon(\Omega)}{\sum_{j \in Z_\varepsilon(\Omega),\xi \in \mathbb{Z}^N}}  C^{j-i,\xi}_{\varepsilon,\frac{\mathrm{d}}{2K}} |u'_\varepsilon(j+\varepsilon\xi)-v'_\varepsilon(j+\varepsilon\xi)|^p\\&+\sum_{i \in Z_\varepsilon(A'\cup B')}\varepsilon^N \underset{j+\varepsilon\xi \in Z_\varepsilon(\Omega)}{\sum_{j \in Z_\varepsilon(\Omega),\xi \in \mathbb{Z}^N}}C^{j-i,\xi}_{\varepsilon,\frac{\mathrm{d}}{2K}}\left(|D^\xi_\varepsilon u'_\varepsilon(j)|^p+ |D^\xi_\varepsilon v'(j)|^p+1\right)\\&+ \sum_{i \in Z_\varepsilon(S_k)}   \varepsilon^N  \underset{j+\varepsilon\xi \in Z_\varepsilon(\Omega)}{\sum_{j \in Z_\varepsilon(\Omega),\xi \in \mathbb{Z}^N}}C^{j-i,\xi}_\varepsilon\left(|D^\xi_\varepsilon u'_\varepsilon(j)|^p+ |D^\xi_\varepsilon v'(j)|^p+1\right)\\& + C\sum_{i \in Z_\varepsilon(S_k)} \varepsilon^N \left(\phi_i^\varepsilon(\{(v'_\varepsilon)_{j+i}\}_{j \in Z_\varepsilon(\Omega_i)}) +\phi_i^\varepsilon(\{(u'_\varepsilon)_{j+i}\}_{j \in Z_\varepsilon(\Omega_i)}) \right).
\end{align*}
Note that $\# \{j \neq k : S_k \cap S_j \neq \emptyset\} \leq 5$. Therefore summing over $k \in \{K,\ldots,2K-1\}$, averaging and taking into account (\ref{Energybound1})--(\ref{convergencebounds1}), (\ref{smoothapproxbound2}) and  Lemma 3.6 in \cite{AC}, we get
\begin{align} \label{Averagebound}
\frac{1}{K} \sum_{k=K}^{2K-1}F_\varepsilon(w^k_\varepsilon,A'\cup B') \leq F_\varepsilon(u_\varepsilon,A) +F_\varepsilon(v_\varepsilon,B) + \frac{C}{K} + (K^p +1) O(\varepsilon).
\end{align}
For any $\varepsilon>0$ there exists $k(\varepsilon) \in \{K,\ldots,2K-1\}$ such that
\begin{align}\label{Goodkeps}
F_\varepsilon(w^{k(\varepsilon)}_\varepsilon,A' \cup B') \leq \frac{1}{K} \sum_{k=K}^{2K-1}F_\varepsilon(w^k_\varepsilon,A'\cup B').
\end{align}
Then, since $w_\varepsilon^{k(\varepsilon)}$ still converges to $u$ in $L^p(\Omega;\mathbb{R}^n)$, by (\ref{Averagebound}) and (\ref{Goodkeps}), letting
$\varepsilon \to 0$ we get
\begin{align*}
F''(u,A' \cup B') \leq F''(u,A) + F(u,B) + \frac{C}{K}.
\end{align*}
Letting $K \to \infty$ we obtain the claim.
\end{proof}
\begin{proposition}\label{InnerregProp} Let $\phi_i^\varepsilon : (\mathbb{R}^n)^{Z_\varepsilon(\Omega)} \to [0,+\infty)$ satisfy {\rm (H2)--(H5)}. Then for any $u\in W^{1,p}(\Omega;\mathbb{R}^n)$ and any $A \in \mathcal{A}(\Omega)$ we have
\begin{align*}
\sup_{A' \subset\subset A} F''(u,A') = F''(u,A).
\end{align*}
\end{proposition}
\begin{proof}
Since $F''(u,\cdot)$ is an increasing set function, it suffices to prove
\begin{align*}
\sup_{A' \subset\subset A} F''(u,A') \geq F''(u,A).
\end{align*}
In order to prove this, we define an extension of the functional $F_\varepsilon$ to a functional $\tilde{F}_\varepsilon$ defined on a bounded, smooth, open set $\tilde{\Omega} \supset \supset \Omega$ such that
\begin{align*}
\tilde{F}_\varepsilon(\tilde{u},A) = F_\varepsilon(u,A)
\end{align*}
for all $A \in \mathcal{A}(\Omega)$ and all $\tilde{u} \in \mathcal{A}_\varepsilon(\tilde{\Omega};\mathbb{R}^n)$ such that $\tilde{u}=u$ in $Z_\varepsilon(\Omega)$ and therefore 
\begin{align} \label{FFtilde}
F''(u,A) = \tilde{F}''(\tilde{u},A) 
\end{align}
for all $A \in \mathcal{A}(\Omega)$, $u \in W^{1,p}(\Omega;\mathbb{R}^n)$ and $\tilde{u} \in W^{1,p}(\Omega;\mathbb{R}^n)$ such that $\tilde{u}=u $ a.e. in $\Omega$. To this end we define $F_\varepsilon : \mathcal{A}_\varepsilon(\tilde{\Omega}) \times \mathcal{A}(\tilde{\Omega}) \to [0,+\infty)$ by
\begin{align*}
\tilde{F}_\varepsilon(u,A) = \sum_{i \in Z_\varepsilon(A)} \varepsilon^N \tilde{\phi}_i^\varepsilon(\{u_{j+i}\}_{j \in Z_\varepsilon(\tilde{\Omega}_i)})
\end{align*}
where $\tilde{\phi}_i^\varepsilon : (\mathbb{R}^n)^{Z_\varepsilon(\tilde{\Omega})} \to [0,+\infty)$ is defined by
\begin{align*}
\tilde{\phi}_i^\varepsilon(\{z_{j+i}\}_{j \in Z_\varepsilon(\Omega_i)}):= \begin{cases}\phi_i^\varepsilon(\{(z\lfloor_\Omega)_{j+i}\})_{j \in Z_\varepsilon(\Omega)} &i \in Z_\varepsilon(\Omega)\\ c\sum_{n=1}^N |D^{e_n}_\varepsilon z(i)|^p &i \in \tilde{\Omega} \setminus \Omega
\end{cases}
\end{align*}
with $c >0$ as in (\ref{coercivitybound}).
Note that $\tilde{\phi}_i^\varepsilon$ satisfies (H2)--(H5). Let $u \in W^{1,p}(\Omega;\mathbb{R}^n)$, extended  to $\tilde{u} \in  W^{1,p}(\tilde{\Omega};\mathbb{R}^n)$. Let $ A \in \mathcal{A}(\Omega)$; for $\delta > 0$ find $A^\delta,A_\delta,B_\delta$ such that $A^\delta \supset\supset A \supset\supset A_\delta \supset \supset A'_\delta \supset \supset B^\delta \supset \supset B_\delta$ and
\begin{align*}
|A^\delta \setminus B_\delta| + || \nabla u||_{L^P(A^\delta \setminus B_\delta;\mathbb{R}^{n\times N})} \leq \delta.
\end{align*}
Applying Proposition \ref{Subad} with $ U = A^\delta \setminus \overline{B}_\delta$, $ V = A_\delta $, $U'=A \setminus \overline{B}^\delta$ and $ V'= A'_\delta $ we have $U' \cup V' =A$ and therefore
\begin{align*}
\tilde{F}''(\tilde{u},A)&\leq \tilde{F}''(\tilde{u},U)=\tilde{F}''(\tilde{u},U' \cup V')\leq  \tilde{F}''(u,U) + \tilde{F}''(u,V) \leq \tilde{F}''(\tilde{u},A_\delta) + \tilde{F}''(\tilde{u}, A^\delta \setminus \overline{B}_\delta) \\&\leq \tilde{F}''(\tilde{u},A_\delta) + C\left(|A^\delta \setminus B_\delta| + || \nabla u||_{L^P(A^\delta \setminus B_\delta;\mathbb{R}^{d \times N})}^p\right) \\& \leq \tilde{F}''(u,A_\delta) + C\delta \leq \sup_{A' \subset\subset A} \tilde{F}''(\tilde{u},A') + C\delta
\end{align*}
Applying (\ref{FFtilde}) to $u, \tilde{u}$ and $A, A' $ we obtain 
\begin{align*}
F''(u,A) \leq \sup_{A' \subset\subset A} F''(u,A') + C\delta.
\end{align*}
The claim follows as $\delta \to 0^+$.
\end{proof}

\begin{proposition}\label{LocalityProp} Let $\phi_i^\varepsilon :(\mathbb{R}^n)^{Z_\varepsilon(\Omega)} \to [0,+\infty)$ satisfy {\rm (H2)--(H5)}. Then for any $A \in \mathcal{A}(\Omega)$ and for any $u,v \in W^{1,p}(\Omega;\mathbb{R}^n)$, such that $u= v$ a.e.~in $A$ we have
\begin{align*}
F''(u,A)=F''(v,A)
\end{align*}
\end{proposition}
\begin{proof}
Thanks to Proposition \ref{InnerregProp}, we may assume that $A \subset\subset \Omega$. We first prove
\begin{align*}
F''(u,A) \geq F''(v,A)
\end{align*}
 Given $\delta >0$ there exist $A_\delta \subset \subset A$ such that
\begin{align*}
|A \setminus \overline{A_\delta}| + || \nabla u||_{L^p(\Omega;\mathbb{R}^{n\times N})}^p \leq \delta
\end{align*}
Let $v_\varepsilon : Z_\varepsilon(\Omega) \to \mathbb{R}^n$, $u_\varepsilon : Z_\varepsilon(\Omega) \to \mathbb{R}^n$ be such that $v_\varepsilon \to v$ and $u_\varepsilon \to u$ in $L^p(\Omega;\mathbb{R}^n)$ and
\begin{align*}
\limsup_{\varepsilon \to 0^+} F_\varepsilon(u_\varepsilon,A) &= F''(u,A) \\
\limsup_{\varepsilon \to 0^+} F_\varepsilon(v_\varepsilon,A \setminus \overline{A_\delta}) = F''(v, A \setminus \overline{A_\delta}) &\leq C \left(|A \setminus \overline{A_\delta}| + || \nabla u||_{L^p(\Omega;\mathbb{R}^{n\times N})}^p\right) \leq C\delta
\end{align*}
Performing the same cut-off construction as in Proposition \ref{Subad} we obtain a function $w_\varepsilon$ converging to $v$ in $L^p(\Omega;\mathbb{R}^n)$ such that for $\varepsilon >0$ small enough we obtain
\begin{align*}
F_\varepsilon(w_\varepsilon,A') \leq F_\varepsilon(u_\varepsilon,A) + F_\varepsilon(v_\varepsilon, A \setminus \overline{A_\delta}) + \frac{C_\delta}{K} + K^pO(\varepsilon)
\end{align*}
for some $A' \subset\subset A$. Taking $\varepsilon \to 0^+$ we obtain
\begin{align*}
F''(v,A') \leq F''(u,A) + \frac{C_\delta}{K} +C\delta
\end{align*} 
Letting $K \to +\infty$ and $\delta \to 0$ we obtain the desired inequality. Exchanging the roles of $u$ and $v$ we obtain the other inequality.
\end{proof}
\begin{proof}[Proof of Theorem {\rm \ref{Compactness}}]  By the compactness property of $\Gamma$-convergence there exists a subsequence $\varepsilon_{j_k} $ of $\varepsilon_j$ such that for any $(u,A) \in W^{1,p}(\Omega;\mathbb{R}^n) \times \mathcal{A}(\Omega)$ there exists
\begin{align*}
\Gamma(L^p)\text{-}\lim_{k } F_{\varepsilon_{j_k}}(u,A) =: F(u,A)
\end{align*}
 (see \cite{BDF} Theorem 10.3). Moreover, by Proposition \ref{CoercivityProp} we have that
 \begin{align*}
\Gamma(L^p)\text{-}\lim_{k } F_{\varepsilon_{j_k}}(u) = +\infty
 \end{align*}
 for any $u \in L^p(\Omega;\mathbb{R}^n) \setminus W^{1,p}(\Omega;\mathbb{R}^n)$. So it suffices to check that for every $(u,A) \in W^{1,p}(\Omega;\mathbb{R}^n) \times \mathcal{A}(\Omega)$, $F(u,A)$ satisfies all the hypothesis of Theorem 2.2 in \cite{AC}. In fact the superaditivity property of $F_\varepsilon(u,\cdot)$ is conserved in the limit. Thus, as an consequence of Propositions (\ref{CoercivityProp})--(\ref{LocalityProp}) and thanks to De Giorgi-Letta Criterion (see \cite{BDF}), hypotheses (i), (ii), (iii) hold true. Moreover, since $F_\varepsilon(u,A)$ is translationally invariant, hypothesis (iv) is satisfied and finally, by the lower semicontinuity property of $\Gamma$-limit, also hypothesis (v) is fulfilled.
\end{proof}

\section{Treatment of Dirichlet boundary data }\label{DBT}
In order to recover the limiting energy density we will establish the next lemma which asserts that our energies still converge if we suitably assign affine boundary conditions. From this, one is able to recover the  value of $f$ in Theorem \ref{Compactness} by a blow-up argument. Given $M \in \mathbb{R}^{n\times N}$,$m \in \mathbb{N}$, $\varepsilon > 0$ and $A\in \mathcal{A}^{reg} (\Omega)$ set
\begin{align}
\mathcal{A}_\varepsilon^{M,m}(A;\mathbb{R}^n) = \left\{u \in \mathcal{A}_\varepsilon(\Omega;\mathbb{R}^n) : u(i) =Mi \text{ if } (i + [-m\varepsilon,m\varepsilon)^N) \cap A^c \neq \emptyset  \right\}
\end{align}
For $M \in \mathbb{R}^{d \times N}, m \in \mathbb{N}$ we define $F_\varepsilon^{M,m} : L^p(\Omega;\mathbb{R}^n) \times \mathcal{A}^{reg}(\Omega) \to [0,+\infty]$ by
\begin{align*}
F^{M,m}_\varepsilon(u,A) = \begin{cases}
F(u,A) &\text{if } u \in \mathcal{A}_\varepsilon^M(A;\mathbb{R}^n) \\
+\infty &\text{otherwise.}
\end{cases}
\end{align*}

\begin{proposition}\label{BoundaryGammaconvergencelemma} Let $\phi_i^\varepsilon : (\mathbb{R}^n)^{Z_\varepsilon(\Omega)} \to [0,+\infty)$ satisfy {\rm(H1)--(H5)}. Let $\varepsilon_{j_k}$ and $f$ be as in Theorem {\rm\ref{Compactness}}. For any $M \in \mathbb{R}^{d \times N}$ and $A \in \mathcal{A}^{reg}(\Omega)$ we set $F^{M} : L^p(\Omega;\mathbb{R}^n) \times \mathcal{A}^{reg}(\Omega) \to [0,+\infty]$ by
\begin{align*}
F^M(u,A) = \begin{cases} \displaystyle\int_A f(x,\nabla u)\mathrm{d}x &\text{if } u-Mx \in W^{1,p}_0(A;\mathbb{R}^n)\\
+\infty &\text{otherwise.}
\end{cases}
\end{align*}
Then for any $M \in \mathbb{R}^{d \times N},m \in \mathbb{N}$ and any $A \in \mathcal{A}^{reg}$ we have that $F^{M,m}_{\varepsilon_{j_k}}(\cdot,A)$ $\Gamma$-converges with respect to the strong $L^p(\Omega;\mathbb{R}^n)$-topology to the functional $F^M(\cdot,A)$.

\end{proposition}
\begin{proof} We only prove the statement for $m=1$, the other cases being done analogously.

We first prove the $\Gamma$-$\liminf$ inequality. Let $\{u_k\}_k \subset \mathcal{A}_{\varepsilon_{j_k}}(\Omega;\mathbb{R}^n) $ converge to $u$ in the $L^p(\Omega;\mathbb{R}^n)$-topology and be such that
\begin{align*}
\liminf_{k \to \infty} F^{M,1}_{\varepsilon_{{j_k}}}(u_k,A) = \lim_{k \to \infty}  F^M_{\varepsilon_{{j_k}}}(u_k,A) < + \infty.
\end{align*}
Since $u_k \in \mathcal{A}_{\varepsilon_{j_k}}^{M,m}(A;\mathbb{R}^n)$ for all $k \in \mathbb{N}$, and by (H3), we have that $u_k \to Mx $ in $L^p(A \setminus \Omega;\mathbb{R}^n)$ and
\begin{align*}
\sup_{\varepsilon > 0} \sum_{n=1}^N \sum_{i \in Z_\varepsilon(\Omega)} \varepsilon^N |D^{e_n}_\varepsilon u_k(i)|^p < +\infty.
\end{align*}
By the same reasoning as in Proposition \ref{CoercivityProp} $u \in W^{1,p}(\Omega;\mathbb{R}^n) $ and $u-Mx \in W^{1,p}_0(A;\mathbb{R}^n)$.
By Theorem \ref{Compactness} we therefore have
\begin{align*}
\liminf_{k \to \infty} F^{M,m}_{\varepsilon_{{j_k}}}(u_k,A)  \geq \liminf_{k \to \infty} F_{\varepsilon_{{j_k}}}(u_k,A) = F^M(u,A).
\end{align*}
To prove the $\Gamma$-$\limsup$ inequality we may first suppose that $\mathrm{supp}(u-Mx) \subset\subset A$. Let $\{u_k\}_k \subset A_{\varepsilon_{j_k}}(\Omega;\mathbb{R}^n)$ converge to $u$ in $L^p(\Omega;\mathbb{R}^n)$ and be such that
\begin{align*}
\limsup_{k \to \infty} F_{\varepsilon_{j_k}}(u_k,A) = F(u,A).
\end{align*} 
Then by reasoning as in the proof of Proposition \ref{InnerregProp} given $\delta > 0$ we can find $A_\delta \subset A$ and suitable cut-off functions $\varphi_k$ with $\mathrm{supp}(u-Mx) \subset\subset \mathrm{supp } \,   \varphi_k \subset\subset A_\delta$ and $|A\setminus A_\delta| < \delta$ such that for 
\begin{align*}
w_k(i) := \varphi_k(i)u_k(i) + (1- \varphi_k(i))Mi
\end{align*}
we have that $w_k $ converges to $u$ in $L^p(\Omega;\mathbb{R}^n)$ and
\begin{align*}
\limsup_{k \to \infty} F_{\varepsilon_{j_k}}(w_k,A) \leq  \limsup_{k \to \infty} F_{\varepsilon_{j_k}}(u_k,A) + \limsup_{k \to \infty}F_{\varepsilon_{j_k}}(Mx,A \setminus A_\delta) + \delta.
\end{align*}
Using (H2) we have that for every $k \in \mathbb{N}$ it holds
\begin{align*}
F_{\varepsilon_{j_k}}(Mx,A \setminus A_\delta) \leq C(|M|^p +1 )|(A\setminus A_\delta)_\varepsilon| \leq C(|M|^p +1 )| \delta.
\end{align*}
By the definition of the $\Gamma$-$\limsup$ we have that
\begin{align*}
\Gamma\text{-}\limsup_{k \to \infty} F^{M,m}_{\varepsilon_{j_k}}(u,A) \leq F^M(u,A) + C\delta.
\end{align*}
Letting $\delta \to 0$ we obtain the desired inequality. The general case follows by a density argument, approximating every function $u \in W^{1,p}(\Omega;\mathbb{R}^n)$ such that $u-Mx \in W^{1,p}_0(A;\mathbb{R}^n)$ strongly in $W^{1,p}(\Omega;\mathbb{R}^n)$ by functions $u_n$ such that $\mathrm{supp}(u_n-Mx) \subset\subset A$ and using the lower semicontinuity of the $\Gamma$-$\limsup$ as well as the continuity of $F(\cdot,A)$ with respect to the strong convergence in $W^{1,p}(\Omega;\mathbb{R}^n)$.
\end{proof}
\begin{remark}\label{Boundaryconvergencelemma} Let $\phi_i^\varepsilon : (\mathbb{R}^n)^{Z_\varepsilon(\Omega)} \to [0,+\infty)$ satisfy (H1)--(H5), and let $\varepsilon_{j_k}$ be as in Theorem \ref{Compactness}. For any $M \in \mathbb{R}^{d \times N}, m\in \mathbb{N}$ and $A \in \mathcal{A}^{reg}(\Omega)$ we have that
\begin{align*}
\lim_{k\to \infty} \inf \left\{F_{\varepsilon_{j_k}}(u,A) : u \in \mathcal{A}_{\varepsilon_{j_k}}^{M,m}(A;\mathbb{R}^n) \right\} =  \inf \left\{F(u,A) : u -Mx \in W^{1,p}_0(A;\mathbb{R}^n)\right\},
\end{align*}
since the functionals $F^M_\varepsilon$ are coercive with respect to the strong $L^p(\Omega;\mathbb{R}^n)$-topology. 

Note first that by extending the functional as in the proof of Proposition \ref{InnerregProp} we can assume that $A\subset\subset \Omega$.  Moreover, by the boundary conditions and by (H3) any sequence $\{u_k\}_k$ satisfying 
\begin{align*}
\sup_{k} F_{\varepsilon_{j_k}}^{M,m}(u_k,A) <+\infty
\end{align*}
satisfies
\begin{align*}
\sup_{k \in \mathbb{N}} \sum_{n=1}^N \sum_{i \in Z_{\varepsilon_{j_k}}(\Omega)} \varepsilon^N |D^{e_n}_{\varepsilon_{j_k}} u_k(i)|^p < +\infty.
\end{align*}
Then by the boundary conditions, Lemma 3.6 in \cite{AC} and the Riesz-Frech\'et-Kolmogorov Theorem there exists a function $u \in L^p(\Omega;\mathbb{R}^n)$ and a subsequence (not relabelled) that converges to $u$. By Proposition \ref{CoercivityProp} we have that $u \in W^{1,p}(\Omega;\mathbb{R}^n)$. Moreover, $u_k \to Mx$ in $L^p(\Omega\setminus A;\mathbb{R}^n)$ and therefore $u-Mx \in W^{1,p}_0(A;\mathbb{R}^n)$.
This implies the coercivity.
\end{remark}

\section{Homogenization}\label{HOM}
We now consider the case where $i \mapsto \phi_i^\varepsilon$ is periodic, though we have to explain what that means in our case, since the interaction energy at every point of the lattice may depend on the whole configuration of the state $\{z_{j+i}\}_{j \in Z_\varepsilon(\Omega_i)}$. This will be done by using a function $\phi_i : (\mathbb{R}^n)^{\mathbb{Z}^N} \to [0,+\infty)$, $ i \in \mathbb{Z}^N$ defined on the entire lattice. 
In order to define the energy density inside $\Omega$ we assume that $\phi_i$ is approximated
by finite-range interaction. More precisely, we suppose that
there exist $\phi_i^k : (\mathbb{R}^n)^{\mathbb{Z}^N} \to [0,+\infty)$, $ i \in \mathbb{Z}^N$ $T$-periodic, satisfying (H1)--(H3) uniformly in $k$ and 

\medskip

($\text{H}_p$4) ({\em locality}) For all $k \in \mathbb{N}$ and for all $z,w \in \mathcal{A}_1(\mathbb{R}^N,\mathbb{R}^n)$ satisfying $z(j) =w(j) $ for all $j \in \mathbb{Z}^N \cap Q_k(i)$ we have 
\begin{align*}
\phi_i^k(\{z_j\}_{j \in \mathbb{Z}^N}) =  \phi_i^k(\{w_{j}\}_{j \in \mathbb{Z}^N}).
\end{align*} 

($\text{H}_p$5) ({\em controlled non-convexity}) There exist $C>0$ and $\{C^{j,\xi}\}_{j \in \mathbb{Z}^N, \xi \in \mathbb{Z}^N}$, $C^{j,\xi} \geq  0$ satisfying 
\begin{align}\label{Assumptions on Cxik}
\sum_{j,\xi \in \mathbb{Z}^N} C^{j,\xi} <+\infty \text{ and we have } \limsup_{k \to \infty}  \sum_{\max\{|\xi|,|j|\} >k} C^{j,\xi} =0
\end{align} 
such that for all $k \in \mathbb{N}$, $z,w \in \mathcal{A}_1(\mathbb{R}^N,\mathbb{R}^n)$ 
 and $\psi$ cut-off functions we have
\begin{align*}
\phi_i^k(\{\psi_{j}z_{j} + (1-\psi_{j})w_{j}\}_{j \in \mathbb{Z}^N}) \leq & C \left( \phi_i^k( \{z_{j}\}_{j \in \mathbb{Z}^N}) + \phi_i^k( \{w_{j}\}_{j \in \mathbb{Z}^N})\right) \\&+R^k_i(z,w,\psi),
\end{align*}
where
\begin{align*}
R^k_i(z,w,\psi) = &\underset{j+\xi \in \mathbb{Z}^N \cap Q_k(0)} {\sum_{j,\xi \in \mathbb{Z}^N}} C^{j,\xi} \Big((\underset{n \in \{1,\ldots,N\}}{\sup_{k \in \mathbb{Z}^N \cap Q_k(0)}} |D^{e_n}_1 \psi(k)|^p+1) |z(j+\xi)-w(j+\xi)|^p\Big) \\+&\underset{j+\xi \in \mathbb{Z}^N \cap Q_k(0)} {\sum_{j,\xi \in \mathbb{Z}^N}} C^{j,\xi}\Big(  |D^\xi_1 z(j)|^p + |D^\xi_1 w(j)|^p +1 \Big).
\end{align*}

\medskip

($\text{H}_p$6) ({\em closeness})  There exist $\{C_{k}^{j,\xi}\}_{k\in \mathbb{N},j \in \mathbb{Z}^N, \xi \in \mathbb{Z}^N}$, $C^{j,\xi}_{k} \geq C^{j,\xi}_{k+1} \geq 0$ satisfying 
\begin{align}\label{Assumptions on Cjk}
 \limsup_{k \to \infty}  \sum_{j,\xi \in \mathbb{Z}^N} C^{j,\xi}_{k} =0
\end{align}
such that For all $z \in \mathcal{A}_1(\mathbb{R}^N;\mathbb{R}^n)$ and $k_1 \leq k_2$ we have that
\begin{align*}
|\phi_i^{k_1}(\{z_{j}\}_{j \in \mathbb{Z}^N}) - \phi_i^{k_2}(\{z_{j}\}_{j \in \mathbb{Z}^N})| \leq \underset{j+\xi \in \mathbb{Z}^N \cap Q_{k_2}(0)}{\sum_{j,\xi \in \mathbb{Z}^N \cap Q_{k_2}(0)}}C^{j,\xi}_{k_1} \left(|D^\xi_1 z(j)|^p +1\right) .
\end{align*}

\medskip

($\text{H}_p$7) ({\em monotonicity}) For every $k \in \mathbb{N}$, for every $i \in \mathbb{Z}^N$ and for every $z \in \mathcal{A}_1(\mathbb{R}^N;\mathbb{R}^n)$ we have
\begin{align}\label{monotonicity}
\phi_i^k(\{z_{j}\}_{j \in \mathbb{Z}^N}) \leq \phi_i^{k+1}(\{z_{j}\}_{j \in \mathbb{Z}^N}), \quad \phi_i^k(\{z_{j}\}_{j \in \mathbb{Z}^N}) \to \phi_i(\{z_{j}\}_{j \in \mathbb{Z}^N}) \text{ as } k \to \infty.
\end{align}

The monotonicity property ($\text{H}_p$7) may seem restrictive at a first sight, but it is not since by the positivity of $\phi^k$ and $\phi$ respectively we may reorder the interactions in a way that we keep only adding positive interactions with increasing $k$. 

For every $i \in Z_\varepsilon(\Omega)$ we define $\phi_i^\varepsilon : (\mathbb{R}^n)^{Z_\varepsilon(\Omega)} \to [0,+\infty)$ by
\begin{align}\label{Periodicitydef}
\phi_i^\varepsilon(\{z_{j}\}_{j \in Z_\varepsilon(\Omega_i)}) = \phi_{\frac{i}{\varepsilon}}^{\lfloor\frac{d_i}{\varepsilon}\rfloor}(\{z^\varepsilon_{j}\}_{j \in \mathbb{Z}^N }),
\end{align}
where $\mathrm{dist}_\infty(\Omega^c,i)=d_i$ and
\begin{align*}
z^\varepsilon(j)= \begin{cases} \frac{z(\varepsilon j)}{\varepsilon} & j \in Q_{\lfloor\frac{d_i}{\varepsilon}\rfloor}(i)\cap \mathbb{Z}^N\\
0 &\text{otherwise.}
\end{cases}
\end{align*}
  Note that (\ref{Periodicitydef}) is well defined due to the  locality property ($\text{H}_p$4) and Moreover, $\phi_i^\varepsilon$ satisfies (H1)--(H5). Those assumptions are made to avoid the dependence of $\phi_i^\varepsilon$ on $ \Omega $ and still include infinite-range interactions.


\begin{theorem}\label{Homogenization} Let $\phi_i^k  : (\mathbb{R}^n)^{\mathbb{Z}^N} \to [0,+\infty)$ satisfy {\rm (H1)--(H3)} and {\rm ($\text{H}_p$4)--($\text{H}_p$7)} and $\phi_i^\varepsilon : (\mathbb{R}^n)^{Z_\varepsilon(\Omega)} \to [0,+\infty)$ be defined by {\rm(\ref{Periodicitydef})}. Then, $F_\varepsilon : L^p(\Omega;\mathbb{R}^n) \to [0,+\infty]$ $\Gamma$-converges with respect to the strong $L^p(\Omega;\mathbb{R}^n)$-topology to the functional $F : L^p(\Omega;\mathbb{R}^n) \to [0,+\infty]$ defined by 
\begin{align*}
F(u) = \begin{cases}\displaystyle\int_\Omega f_{\mathrm{hom}}(\nabla u)\mathrm{d}x &\text{if } u \in W^{1,p}(\Omega;\mathbb{R}^n) \\
+\infty &\text{otherwise,}
\end{cases}
\end{align*}
where $f_{\mathrm{hom}}: \mathbb{R}^{d \times N}\to [0,\infty)$ is given by
\begin{align}\label{fhom}
f_{\mathrm{hom}}(M) = \lim_{L \to \infty} \frac{1}{L^N} \inf \Big\{\sum_{i \in \mathbb{Z}^N \cap Q_L} \phi_{i}(\{z_{j+i}\}_{j \in \mathbb{Z}^N }) : z \in \mathcal{A}^{M,\lfloor\sqrt{L}\rfloor}_1(Q_L;\mathbb{R}^n) \Big\},
\end{align}
where
\begin{align*}
\mathcal{A}_\varepsilon^{M,m}(Q_L;\mathbb{R}^n) = \left\{u \in \mathcal{A}_\varepsilon(\mathbb{R}^N;\mathbb{R}^n) : u(i) =Mi \text{ if } (i + [-m\varepsilon,m\varepsilon)^N) \cap Q_L^c \neq \emptyset  \right\}.
\end{align*}
\end{theorem}
\begin{remark}\label{fhomrecovery} Note that in Theorem \ref{Homogenization} we have that the whole sequence $F_\varepsilon$ $\Gamma$-converges to the limit functional $F$. We fix the boundary conditions of the admissible test functions on a boundary layer of width $\lfloor\sqrt{L}\rfloor$ in order to have the boundary effects negligible while still being able to use a subadditivity argument in order to prove the existence of the limit in (\ref{fhom}). Arguing as in the proof of Proposition \ref{fhomprop} to show that the error goes to $0$ when substituting $\phi_i^k$ with $\phi_i$, and using the fact that the limit energy density is quasi-convex,
we also have
\begin{align*}
f_{\mathrm{hom}}(M) = \lim_{L \to \infty} \frac{1}{L^N} \inf \Big\{\sum_{i \in \mathbb{Z}^N \cap Q_L} \phi_{i}(\{z_{j+i}\}_{j \in \mathbb{Z}^N }) : z \in \mathcal{A}^{M,m}_1(Q_L;\mathbb{R}^n) \Big\}
\end{align*}
for all $m \in \mathbb{N}$ and all $M \in \mathbb{R}^{d \times N}$.
\end{remark}

\begin{proof} By Theorem (\ref{Compactness}) for every sequence $\varepsilon_j$ there exists a subsequence $ \varepsilon_{j_k} $ such that $F_{\varepsilon_{j_k}}$ $\Gamma$-converges to a functional $F$ such that for any $u \in W^{1,p}(\Omega;\mathbb{R}^n)$ and every $A \in \mathcal{A}(\Omega)$ we have
\begin{align*}
\Gamma\text{-}\lim_{k\to \infty} F_{\varepsilon_{j_k}}(u,A) =\int_A f(x,\nabla u)\mathrm{d}x.
\end{align*}
By the Urysohn property of $\Gamma$-convergence the theorem is proved if we show that $f$ does not depend on $x$ and $f = f_{\mathrm{hom}}$. To prove the first claim it suffices to show that
\begin{align*}
F(Mx,Q_\rho(z))= F(Mx,Q_\rho(y))
\end{align*}
for all $M \in \mathbb{R}^{d \times N}$, $z,y \in \Omega$ and $\rho >0$ such that $Q_\rho(z) \cup Q_\rho(y) \subset \Omega$. By symmetry it suffices to prove
\begin{align*}
F(Mx,Q_\rho(z)) \leq F(Mx,Q_\rho(y)).
\end{align*}
By the inner-regularity property it suffices to prove for any $\rho' < \rho$
\begin{align*}
F(Mx,Q_{\rho'}(z))\leq  F(Mx,Q_\rho(y)).
\end{align*}
Let $v_k \in \mathcal{A}_{\varepsilon_{j_k}}(\Omega;\mathbb{R}^n)$ be such that $v_k \to Mx$ in $L^p(\Omega;\mathbb{R}^n)$ and such that
\begin{align*}
\lim_{k \to \infty} F_{\varepsilon_{j_k}}(v_k,Q_\rho(y)) = F(Mx,Q_\rho(y)).
\end{align*}
Let $\varphi \in C^\infty(\Omega) $ be a cut-off function such that $0 \leq \varphi \leq 1$
\begin{align*}
\mathrm{supp}(\varphi) \subset\subset Q_\rho(z), \quad  Q_{\rho'}(z) \subset \subset\{\varphi = 1\}  \text{ and } ||\nabla \varphi ||_\infty \leq \frac{C}{\rho-\rho'}.
\end{align*}
For $k \in \mathbb{N}$ define $u_k \in \mathcal{A}_{\varepsilon_{j_k}}(\Omega;\mathbb{R}^n)$ by
\begin{align*}
u_k(i) =  \varphi(i)\left( v_k\Big(i + \varepsilon_{j_k} T \lfloor \frac{y-z}{T\varepsilon_{j_k}} \rfloor \Big) + M(z-y)\right) +(1-\varphi(i)) Mi.
\end{align*}
Thus by the periodicity assumption and the locality property we have that
\begin{align*}
\sum_{i \in Z_{\varepsilon_{j_k}}(Q_{\rho'}(z))}\varepsilon_{j_k}^N\phi_i^{\varepsilon_{j_k}} (\{(u_k)_{j + i}\}_{ j \in Z_{\varepsilon_{j_k}} (\Omega_i)}) \leq  \sum_{i \in Z_{\varepsilon_{j_k}}(Q_{\rho}(y))}\varepsilon_{j_k}^N\phi_{i}^{\varepsilon_{j_k}} (\{(v_k)_{j + i}\}_{ j \in Z_{\varepsilon_{j_k}} (\Omega_{i})}) + O({\varepsilon_{j_k}}).
\end{align*}

Therefore, we obtain
\begin{align*}
F(Mx,Q_{\rho'}(z)) \leq \liminf_{k \to \infty} F_{\varepsilon_{j_k}}(u_k,Q_{\rho'}(z)) \leq  \liminf_{k \to \infty} F_{\varepsilon_{j_k}}(u_k,Q_{\rho}(y)) = F(Mx,Q_{\rho}(y)).
\end{align*}
In order to obtain that $f =f_{\mathrm{hom}}$ we note that by the lower semicontinuity with respect to the strong $L^p(\Omega;\mathbb{R}^n)$-topology and the coercivity of $F$ we obtain that $F$ is lower semicontinuous with respect to the weak $W^{1,p}(\Omega;\mathbb{R}^n)$-topology and hence $f$ is quasiconvex. By the growth properties of $f$ and Remark \ref{Boundaryconvergencelemma} we obtain for $Q=Q_\rho(x_0) \subset \subset \Omega $ 
\begin{align*}
f(M) &= \frac{1}{\rho^N}\inf\Big\{ \int_Q f(\nabla u )\mathrm{d}x : u - Mx \in W^{1,p}_0(Q;\mathbb{R}^n) \Big\} \\&=\frac{1}{\rho^N}\inf\Big\{ F(u,Q) : u - Mx \in W^{1,p}_0(Q;\mathbb{R}^n) \Big\} \\&= \lim_{m \to \infty} \lim_{k \to \infty} \frac{1}{\rho^N}\inf\Big\{ F_{\varepsilon_{j_k}}(u,Q) : u \in \mathcal{A}^{M,m}_{\varepsilon_{j_k}}(Q;\mathbb{R}^n) \Big\}\\&=f_{\mathrm{hom}}(M).
\end{align*}
Where the last inequality follows from the next proposition.
\end{proof}
\begin{proposition}\label{fhomprop}  Let $\phi_i^k  : (\mathbb{R}^n)^{\mathbb{Z}^N} \to [0,+\infty)$ satisfy {\rm(H1)--(H3)} and {\rm($\text{H}_p$4)--($\text{H}_p$7)}, and $\phi_i^\varepsilon : (\mathbb{R}^n)^{Z_\varepsilon(\Omega)} \to [0,+\infty)$ be defined by {\rm(\ref{Periodicitydef})}. Then
\begin{align*}
 f_{\mathrm{hom}}(M)= \lim_{m \to \infty}\lim_{k \to \infty} \frac{1}{\rho^N}\inf\Big\{ F_{\varepsilon_{j_k}}(u,Q) : u \in \mathcal{A}^{M,m}_{\varepsilon_{j_k}}(Q;\mathbb{R}^n) \Big\}
\end{align*}
for all $M \in \mathbb{R}^{n \times N}$.
\end{proposition}
\begin{proof} Without loss of generality, assume $x_0=0$. We perform a change of variables 
\begin{align*}
i' = \frac{i}{\varepsilon_{j_k}}, \quad \tilde{u}(i') = \frac{1}{\varepsilon_{j_k}} u(\varepsilon_{j_k} i'), \quad L_k = \frac{\rho}{\varepsilon_{j_k}}.
\end{align*}
Set $d_{i'}^k= \mathrm{dist}(\frac{1}{\varepsilon_{j_k}}\Omega^c,i')$. We obtain
\begin{align*}
&\lim_{m \to \infty}\lim_{k \to \infty} \frac{1}{\rho^N}\inf\Big\{ F_{\varepsilon_{j_k}}(u,Q) : u \in \mathcal{A}^{M,m}_{\varepsilon_{j_k}}(Q;\mathbb{R}^n) \Big\}  \\= &\lim_{m \to \infty}\lim_{k \to \infty} \frac{1}{L_k^N} \inf\Big\{ \sum_{i' \in \mathbb{Z}^N \cap Q_L} \phi_{i'}^{\lfloor d_{i'}^k\rfloor}(\{\tilde{u}_{j+i'}\}_{j \in \mathbb{Z}^N}) : \tilde{u} \in \mathcal{A}_1^{M,m}(Q_{L_k};\mathbb{R}^n)\Big\}.
\end{align*} 
By the monotonicity property and (H2) we have that
\begin{align*}
C(|M|^p+1)&\geq \lim_{m \to \infty}\lim_{k \to \infty} \frac{1}{L_k^N} \inf\Big\{ \sum_{i' \in \mathbb{Z}^N \cap Q_L} \phi_{i'}(\{\tilde{u}_{j+i'}\}_{j \in \mathbb{Z}^N})  : \tilde{u} \in \mathcal{A}_1^{M,m}(Q_L;\mathbb{R}^n)\Big\}\\&\geq \lim_{m \to \infty}\lim_{k \to \infty} \frac{1}{L_k^N} \inf\Big\{ \sum_{i' \in \mathbb{Z}^N \cap Q_L} \phi_{i'}^{\lfloor d_{i'}^k\rfloor}(\{\tilde{u}_{j+i'}\}_{j \in \mathbb{Z}^N})  : \tilde{u} \in \mathcal{A}_1^{M,m}(Q_L;\mathbb{R}^n)\Big\}.
\end{align*}
On the other hand, let $u_k \in \mathcal{A}_1^{M,m}(Q_L;\mathbb{R}^n)$ be such that
\begin{align*}
 \sum_{i' \in \mathbb{Z}^N \cap Q_{L_k}} \phi_{i'}^{\lfloor d_{i'}^k\rfloor}(\{(u_k)_{j+i'}\}_{j \in \mathbb{Z}^N})  \leq\inf\Big\{ \sum_{i' \in \mathbb{Z}^N \cap Q_{L_k}} \phi_{i'}^{\lfloor d_{i'}^k\rfloor}(\{\tilde{u}_{j+i'}\}_{j \in \mathbb{Z}^N})  : \tilde{u} \in \mathcal{A}_1^{M,m}(Q_{L_k};\mathbb{R}^n)\Big\} +\frac{1}{k}
\end{align*}
Now by ($\text{H}_p$6) and setting $\displaystyle d_k = \lfloor\frac{\mathrm{dist}(Q,\Omega^c)}{\varepsilon_{j_k}} \rfloor$ we obtain $d_k \to \infty$, since $Q \subset\subset \Omega$, and
\begin{align*}
\sum_{i' \in \mathbb{Z}^N \cap Q_{L_k}} \phi_{i'}(\{(u_k)_{j+i'}\}_{j \in \mathbb{Z}^N})  \leq   \sum_{i' \in \mathbb{Z}^N \cap Q_{L_k}} \Big( \phi_{i'}^{\lfloor d_{i'}^k\rfloor}(\{(u_k)_{j+i'}\}_{j \in \mathbb{Z}^N}) + \sum_{j,\xi \in \mathbb{Z}^N} C^{j-i',\xi}_{d_k}( |D^\xi_1 u_k(j)|^p +1) \Big).
\end{align*}
We have that either $j, j + \xi \in \mathbb{Z}^N \setminus Q_{L_k}(0) $ in which case $|D_1^\xi u_k|^p \leq |M|^p $ or $\{j, j+\xi\} \cap Q_{L_k}(0)\neq \emptyset $. Now if $j,j+\xi \in Q_{L_k}(0)$, by [\cite{AC},Lemma 3.6] and (H2), we have that
\begin{align}\label{Bound1}
\nonumber\underset{j,j+\xi \in Q_{L_k}(0)}{\sum_{j \in \mathbb{Z}^N}} |D^\xi_1 u_k(j)|^p &\leq C \sum_{n=1}^N \sum_{j \in \mathbb{Z}^N \cap Q_{L_k}(0)}  |D^{e_n}_1 u_k(j)|^p \\& \leq C \sum_{j \in \mathbb{Z}^N \cap Q_{L_k}(0)} \phi_{j}^{ d_k}(\{(u_k)_{j'+j}\}_{j' \in \mathbb{Z}^N}) \leq C(|M|^p+1) L_k^N.
\end{align}
Now either $j \in Q_{L_k}(0)$, $j+\xi \notin Q_{L_k}(0) $ or  $j \notin Q_{L_k}(0)$, $j+\xi \in Q_{L_k}(0) $. We only deal with the first case, the second one being done analogously. Now if $|\xi|_\infty \leq L_k$, by (H2) and using the boundary conditions, we have that
\begin{align} \label{Bound2}
\nonumber\sum_{j \in \mathbb{Z}^N} |D^\xi_1 u_k(j)|^p  &\leq  \underset{j,j+\xi \in Q_{2L_k}(0)}{\sum_{j \in \mathbb{Z}^N}} |D^\xi_1 u_k(j)|^p \leq C \sum_{n=1}^N \sum_{j \in \mathbb{Z}^N \cap Q_{2L_k}(0)} |D^{e_n}_1 u_k(j)|^p\\&\nonumber\leq C \sum_{j \in \mathbb{Z}^N \cap Q_{L_k}(0)} \phi_{j}^{ d_k}(\{(u_k)_{j'+j}\}_{j' \in \mathbb{Z}^N}) + \sum_{j \in \mathbb{Z}^N \cap Q_{2L_k}(0) \setminus Q_{L_k}(0)} |D^{e_n}_1 u_k(j)|^p \\&\nonumber\leq C \sum_{j \in \mathbb{Z}^N \cap Q_{L_k}(0)} \phi_{j}^{ d_k}(\{(u_k)_{j'+j}\}_{j' \in \mathbb{Z}^N})  + C L_k^N |M|^p \\&\leq C (|M|^p+1) L_k^N.
\end{align}
If $|\xi|_\infty > L_k$ for every $j$ we choose a path $\gamma_\xi^j = (j_h)_{h=1}^{||\xi||_1+1}\subset \mathbb{Z}^N$ by defining 
\begin{align*}
j_{||\xi||_1+1} = j +\xi, j_1 = j, j_{h+1} =j_{h}+e_{n(h)}, e_{n(h)} = \mathrm{sign}(\xi_k) e_k \text{ if } 1+\sum_{n=1}^{k-1} | \xi_n | \leq h \leq \sum_{n=1}^{k} | \xi_n |.
\end{align*}
For this path it holds
\begin{align*}
|D_1^\xi u(j)|^p \leq  \frac{C(p,N)}{||\xi||_1}\sum_{h=1}^{||\xi||_1} |D^{e_n(h) }_1 u(j_h)|^p.
\end{align*}
Now for every $i \in\mathbb{Z}^N$ and for every $n \in \{1,\ldots,N\}$ we set
\begin{align*}
N_{i,n}^{\xi,k} = \Big\{ j  \in Q_{L_k}(0) : \, &\exists h \in \{1,\ldots,|\xi_1|\}, n \in \{1,\ldots,N\}\\&\text{such that } i = j_h \in \gamma^\xi_j \text{ and } e_{n(h)}= \mathrm{sign}(\xi_n) e_n  \}.
\end{align*}
We have that $ \# N_{i,n}^{\xi,k} \leq L_k $ for $i \in \mathbb{Z}^N \cap Q_{L_k}(0)$, using $|D^{e_n}_1 u_k(i)| \leq |M| $ for every $i \in \mathbb{Z}^N \setminus Q_{L_k}(0)$ and using Fubini's Theorem we obtain
\begin{align} \label{Bound3}
\nonumber\sum_{j \in \mathbb{Z}^N \cap Q_{L_k}(0)}|D^\xi_1 u_k(j)|^p &\leq \frac{C}{||\xi||_1}\sum_{j \in \mathbb{Z}^N \cap Q_{L_k}(0)}\sum_{h=1}^{||\xi||_1} |D^{e_n(h) }_1 u_k(j_h)|^p \\&\nonumber\leq  \frac{C}{||\xi||_1}\sum_{n=1}^N\sum_{i \in \mathbb{Z}^N \cap Q_{L_k}(0) } \# N_{i,n}^{\xi,k} |D^{e_n}_1 u_k(i)|^p + |M|^p L_k^N \\& \nonumber\leq C\sum_{n=1}^N\sum_{i \in \mathbb{Z}^N \cap Q_{L_k}(0) }  |D^{e_n}_1 u_k(i)|^p + |M|^p L_k^N  \\&\nonumber \leq C\sum_{j \in \mathbb{Z}^N \cap Q_{L_k}(0)} \phi_{i}^{ d_k}(\{(u_k)_{j+i}\}_{j \in \mathbb{Z}^N}) + |M|^p L_k^N \\&\leq C(|M|^p+1) L_k^N.
\end{align}
Now if $j,j+\xi \in Q_{L_k}(0)$, using Fubini's Theorem and (\ref{Bound1}), we obtain
\begin{align} \label{Bound11}
 \nonumber\sum_{i' \in \mathbb{Z}^N \cap Q_{L_k}(0)} \underset{j,j+\xi \in Q_{L_k}(0)}{\sum_{ j,\xi \in \mathbb{Z}^N}}C^{j-i',\xi}_{d_k} |D^\xi_1 u_k(j)|^p &\leq   \sum_{i',\xi \in \mathbb{Z}^N } C^{j-i',\xi}_{d_k} \underset{j+\xi \in Q_{L_k}(0)}{\sum_{j \in \mathbb{Z}^N \cap Q_{L_k}(0)}}|D^\xi_1 u_k(j)|^p \\&\leq C L_k^N \sum_{i',\xi \in \mathbb{Z}^N} C^{j-i',\xi}_{d_k} (|M|^p+1).
\end{align}

Now if $j \in Q_{L_k}(0)$ $|\xi|_\infty \leq L_k$, using Fubini's Theorem and (\ref{Bound2}), we obtain
\begin{align} \label{Bound12}
 \nonumber\sum_{i' \in \mathbb{Z}^N \cap Q_{L_k}(0)} \underset{|\xi|_\infty \leq L_k}{\underset{j \in Q_{L_k}(0)}{\sum_{ j,\xi \in \mathbb{Z}^N}}}C^{j-i',\xi}_{d_k} |D^\xi_1 u_k(j)|^p &\leq   \underset{|\xi|_\infty \leq L_k}{\sum_{i',\xi \in \mathbb{Z}^N }} C^{j-i',\xi}_{d_k}\sum_{j \in \mathbb{Z}^N \cap Q_{L_k}(0)}|D^\xi_1 u_k(j)|^p \\&\leq C L_k^N \sum_{i',\xi \in \mathbb{Z}^N} C^{j-i',\xi}_{d_k} (|M|^p+1).
\end{align}

If $j \in Q_{L_k}(0)$ $|\xi|_\infty > L_k$, using Fubini's Theorem and (\ref{Bound3}), we obtain
\begin{align} \label{Bound13}
 \nonumber\sum_{i' \in \mathbb{Z}^N \cap Q_{L_k}(0)} \underset{|\xi|_\infty > L_k}{\underset{j \in Q_{L_k}(0)}{\sum_{ j,\xi \in \mathbb{Z}^N}}}C^{j-i',\xi}_{d_k} |D^\xi_1 u_k(j)|^p &\leq   \underset{|\xi|_\infty > L_k}{\sum_{i',\xi \in \mathbb{Z}^N }} C^{j-i',\xi}_{d_k}\sum_{j \in \mathbb{Z}^N \cap Q_{L_k}(0)}|D^\xi_1 u_k(j)|^p \\&\leq C L_k^N \sum_{i',\xi \in \mathbb{Z}^N} C^{j-i',\xi}_{d_k} (|M|^p+1).
\end{align}
Now, dividing by $L_k^N$, using (\ref{Assumptions on Cjk}),(\ref{Bound11})--(\ref{Bound13}) and taking the limit as $k \to \infty$ , we obtain
\begin{align*}
\lim_{k\to \infty} \frac{1}{L_k^N}\sum_{i' \in \mathbb{Z}^N \cap Q_{L_k}}\sum_{j,\xi \in \mathbb{Z}^N} C^{j-i',\xi}_{d_k}( |D^\xi_1 u_k(j)|^p +1)=0
\end{align*}
It remains to show that the limit (\ref{fhom}) exists and
\begin{align}\label{equality}
f_{\mathrm{hom}}(M) = \lim_{m \to \infty}\lim_{L \to \infty} \frac{1}{L^N} \inf\Big\{\sum_{i \in \mathbb{Z}^N \cap Q_L(0)} \phi_i(\{z_{j+i}\}_{j \in \mathbb{Z}^N}) : z \in \mathcal{A}_1^{M,m}(Q_L;\mathbb{R}^n)\Big\}.
\end{align}
Since $\mathcal{A}_1^{M,\lfloor \sqrt{L}\rfloor}(Q_L;\mathbb{R}^n) \subset  \mathcal{A}_1^{M,m}(Q_L;\mathbb{R}^n)$ we have that
\begin{align*}
f_{\mathrm{hom}}(M) \geq \lim_{m \to \infty} \lim_{L \to \infty} \frac{1}{L^N} \inf\Big\{\sum_{i \in \mathbb{Z}^N \cap Q_L(0)} \phi_i(\{z_{j+i}\}_{j \in \mathbb{Z}^N}) : z \in \mathcal{A}_1^{M,m}(Q_L;\mathbb{R}^n)\Big\}.
\end{align*}
On the other hand, for every $u_L \in \mathcal{A}_1^{M,m}(Q_L;\mathbb{R}^n)$, also $ u_L \in \mathcal{A}_1^{M,\lfloor \sqrt{L+\sqrt{L}}\rfloor}(Q_{L+\lfloor\sqrt{L}\rfloor};\mathbb{R}^n)$, so that for $\tilde{L}= L+\lfloor\sqrt{L}\rfloor$ we have
\begin{eqnarray*}
\sum_{i \in \mathbb{Z}^N \cap Q_{\tilde{L}}(0)} \phi_i(\{(u_L)_{j+i}\}_{j \in \mathbb{Z}^N})& =&\sum_{i \in \mathbb{Z}^N \cap Q_{L}(0)} \phi_i(\{(u_L)_{j+i}\}_{j \in \mathbb{Z}^N})\\
&&\qquad + \sum_{i \in \mathbb{Z}^N \cap (Q_{\tilde{L}}(0) \setminus Q_L(0))} \phi_i(\{(u_L)_{j+i}\}_{j \in \mathbb{Z}^N}).
\end{eqnarray*}
Note that $\lim_{L \to \infty} \frac{\tilde{L}}{L}=1$ and therefore we are done if we can show that 
\begin{align*}
\frac{1}{L^N}\sum_{i \in \mathbb{Z}^N \cap (Q_{\tilde{L}}(0) \setminus Q_L(0))} \phi_i(\{(u_L)_{j+i}\}_{j \in \mathbb{Z}^N}) \to 0 
\end{align*}
as $L \to \infty$ and then $m \to \infty$. By the locality property ($\text{H}_p$4) and the boundary conditions we have for all $i \in \mathbb{Z}^N \cap (Q_{\tilde{L}}(0) \setminus Q_L(0))$
\begin{align*}
\phi_i(\{(u_L)_{j+i}\}_{j \in \mathbb{Z}^N}) &\leq\phi_i(\{Mx_{j+i}\}_{j \in \mathbb{Z}^N})  + \sum_{j,\xi \in \mathbb{Z}^N} C^{j-i,\xi}_m(|D^\xi_1 u_L(j)|^p +1) \\&\leq C(|M|^p+1) + \sum_{j,\xi \in \mathbb{Z}^N} C^{j-i,\xi}_m(|D^\xi_1 u_L(j)|^p +1).
\end{align*}
Using similar arguments as for (\ref{Bound11})--(\ref{Bound13}) we obtain
\begin{align}\label{Boundaryzero}
\frac{1}{L^N}\sum_{i \in \mathbb{Z}^N \cap (Q_{\tilde{L}}(0) \setminus Q_L(0)) }\sum_{j,\xi \in \mathbb{Z}^N} C^{j-i,\xi}_m(|D^\xi_1 u_L(j)|^p +1) \to 0
\end{align}
as $L \to \infty$ and then $m \to \infty$ and hence (\ref{equality}). 
We are done if we show that the limit in the definition of (\ref{fhom}) exists. To this end set
\begin{align*}
 F_L(M)=\frac{1}{L^N} \inf \Big\{\sum_{i \in \mathbb{Z}^N \cap Q_L} \phi_{i}(\{z_{j+i}\}_{j \in \mathbb{Z}^N }) : z \in \mathcal{A}^{M,\sqrt{L}}_1(Q_L;\mathbb{R}^n) \Big\}.
\end{align*}
Let $L \in \mathbb{N}$ and let $k \in \mathbb{N}$ be such that $kT \leq L \leq (k+1)T$. For any  $ u \in \mathcal{A}_1^{M,\lfloor \sqrt{L}\rfloor}(Q_{L};\mathbb{R}^n)$ we have that $ u \in \mathcal{A}_1^{M,\lfloor \sqrt{(k+1)T}\rfloor}(Q_{(k+1)T};\mathbb{R}^n)$ and
\begin{align*}
\frac{1}{L^N}\sum_{i \in \mathbb{Z}^N \cap Q_{(k+1)T}(0)} \phi_i(\{u_{j+i}\}_{j \in \mathbb{Z}^N}) \leq &\frac{1}{L^N}\sum_{i \in \mathbb{Z}^N \cap Q_{L}(0)} \phi_i(\{u_{j+i}\}_{j \in \mathbb{Z}^N}) \\& +\frac{1}{L^N}\sum_{i \in \mathbb{Z}^N \cap (Q_{(k+1)T}(0) \setminus Q_L(0)} \phi_i(\{u_{j+i}\}_{j \in \mathbb{Z}^N}),
\end{align*}
where the last term tends to $0$ as $L \to \infty$, again using similar arguments as to prove (\ref{Boundaryzero}). Noting that for every $k \in \mathbb{N}$ the function $u  \in \mathcal{A}_1^{M,\lfloor \sqrt{kT}\rfloor}(Q_{kT};\mathbb{R}^n)$ can also be used as a test function  $ u \in \mathcal{A}_1^{M,\lfloor \sqrt{L}\rfloor}(Q_{L};\mathbb{R}^n)$ in the minimum on $Q_L$ we obtain that
\begin{align*}
\lim_{k \to \infty} F_{kT}(M) = \lim_{L \to \infty} F_L(M).
\end{align*} 
Hence,  we can assume that $L,S \in T\mathbb{N}$, $1<<L<<S$ and $u_L \in \mathcal{A}_1^{M,\lfloor \sqrt{L}\rfloor}(Q_{L};\mathbb{R}^n)$ be such that
\begin{align*}
\frac{1}{L^N}\sum_{i \in \mathbb{Z}^N \cap Q_{L}(0)} \phi_i(\{(u_L)_{j+i}\}_{j \in \mathbb{Z}^N}) \leq F_L(M) +\frac{1}{L}.
\end{align*}
We define $v_S \in \mathcal{A}_1^{M,\lfloor \sqrt{S}\rfloor}(Q_{S};\mathbb{R}^n)$ by
\begin{align*}
v_S(i) = \begin{cases} u_L(i-L k) + L M k &\text{if } i  \in  L k + Q_{L}(0), k \in \{-\frac{1}{2}\lfloor\frac{S-\sqrt{S}}{L}\rfloor,\ldots, \frac{1}{2}\lfloor\frac{S-\sqrt{S}}{L}\rfloor\}^N \\
Mi &\text{otherwise.}
\end{cases}
\end{align*}
By the periodicity assumption and (H4) we have that
\begin{align*}
F_S(M)&\leq \frac{1}{S^N}\sum_{i \in \mathbb{Z}^N \cap Q_{S}(0)} \phi_i(\{(v_S)_{j+i}\}_{j \in \mathbb{Z}^N}) \\&= \frac{L^N}{S^N} \sum_{k \in \{-\frac{1}{2}\lfloor\frac{S-\sqrt{S}}{L}\rfloor,\ldots, \frac{1}{2}\lfloor\frac{S-\sqrt{S}}{L}\rfloor\}^N }\frac{1}{L^N}\sum_{i \in \mathbb{Z}^N \cap Q_{L}(0)}  \phi_{i+kL}(\{(u_L)_{j+i-kL}\}_{j \in \mathbb{Z}^N}) \\&\leq \frac{L^N}{S^N} \Big\lfloor\frac{S-\sqrt{S}}{L}\Big\rfloor^N \frac{1}{L^N}\sum_{i \in \mathbb{Z}^N \cap Q_{L}(0)}  \phi_{i}(\{(u_L)_{j+i}\}_{j \in \mathbb{Z}^N})\\
&\qquad\qquad+ \frac{1}{S^N}\sum_{i \in Q_S(0)} \sum_{j,\xi \in \mathbb{Z}^N}C^{j-i}_{\sqrt{L}} (|D^\xi_1 v_S(j)|^p+1)\\&\leq  \frac{L^N}{S^N} \Big\lfloor\frac{S-\sqrt{S}}{L}\Big\rfloor^N \frac{1}{L^N} F_L(M)+ \frac{1}{S^N}\sum_{i \in Q_S(0)} \sum_{j,\xi \in \mathbb{Z}^N}C^{j-i}_{\sqrt{L}} (|D^\xi_1 v_S(j)|^p+1).
\end{align*}
Now, again using the same arguments as for (\ref{Bound11})--(\ref{Bound13}), we obtain
\begin{align*}
\limsup_{L \to \infty}\limsup_{S \to \infty}\frac{1}{S^N}\sum_{i \in Q_S(0)} \sum_{j,\xi \in \mathbb{Z}^N}C^{j-i}_{\sqrt{L}} (|D^\xi_1 v_S(j)|^p+1)=0
\end{align*}
and therefore, noting that
$
\lim\limits_{L \to \infty}\lim\limits_{S \to \infty}\frac{L^N}{S^N} \Big\lfloor\frac{S-\sqrt{S}}{L}\Big\rfloor^N =1
$,
we get
$ \displaystyle
\limsup_{S \to \infty} F_S(M) \leq \liminf_{L \to \infty} F_L(M)
$
and the claim follows.
\end{proof}

\section{Examples}\label{EXA}
\subsection{The discrete determinant} An example of interactions that can be taken into account with our type of energies are discrete determinants. For $z \in \mathcal{A}_\varepsilon(\Omega;\mathbb{R}^n)$ we define
\begin{align*}
\phi_i^\varepsilon(\{z_{j}\}_{j \in Z_\varepsilon(\Omega_i)}) =\sum_{\xi_1,\ldots , \xi_n \in \mathbb{Z}^N}g^\varepsilon_{\xi_1,\ldots,\xi_n}(\det(D^{\xi_1}_\varepsilon z(0),\ldots , D^{\xi_n}_\varepsilon z(0))) + \sum_{n=1}^N |D_\varepsilon^{e_n} z(0)|^p,
\end{align*}
where $g_{\xi_1,\ldots,\xi_n}^\varepsilon : \mathbb{R} \to [0,\infty)$ satisfy
\begin{align*}
g^\varepsilon_{\xi_1,\ldots,\xi_n}(z) \leq C_{\xi_1,\ldots,\xi_n}(|z|^{\frac{p}{n}}+1)
\end{align*}
 and $C_{\xi_1,\ldots,\xi_n}>0$ satisfy
\begin{align*}
\sum_{\xi_1,\ldots,\xi_n \in \mathbb{Z}^N} C_{\xi_1,\ldots,\xi_n} < +\infty.
\end{align*}
(H1) follows, since $\phi_i^\varepsilon$ does only depend on its difference quotients. Note that by Hadamard's Inequality, the Geometric-Arithmetic mean Inequality and convexity we have
\begin{align*}
|\det(D^{\xi_1}_\varepsilon z(0),\ldots , D^{\xi_n}_\varepsilon z(0))|^{\frac{p}{n}} \leq \Big|\prod_{j=1}^n |D^{\xi_j}_\varepsilon z(0)|^{\frac{1}{n}} \Big|^p \leq \Big|\frac{1}{n}\sum_{j=1}^n |D^{\xi_j}_\varepsilon z(0)| \Big|^p \leq \frac{1}{n}\sum_{j=1}^n |D^{\xi_j}_\varepsilon z(0)|^p.
\end{align*}
Recall $ \displaystyle
\left|\frac{M(i+\varepsilon\xi)-Mi}{\varepsilon|\xi|}\right| \leq |M|
$
and therefore
\begin{align*}
|\det(D^{\xi_1}_\varepsilon z(0),\ldots , D^{\xi_n}_\varepsilon z(0))|^{\frac{p}{n}} \leq |M|^p
\end{align*}
and by summing over $\xi_1, \ldots, \xi_n \in \mathbb{Z}^N$ (H2) follows. (H3) follows since we have exactly the coercivity term in the definition of $\phi_i^\varepsilon$ and the first term is positive.
For $\delta>0$ and $z(j) =w(j) $ in $Z_\varepsilon(Q_\delta(i))$ we have that
\begin{align*}
\phi_i^\varepsilon(\{z_{j}\}_{j \in Z_\varepsilon(\Omega_i)}) \leq \phi_i^\varepsilon(\{w_{j}\}_{j \in Z_\varepsilon(\Omega_i)}) + \underset{\varepsilon|\xi_i|_\infty>\delta}{\sum_{\xi_1,\ldots,\xi_n\in \mathbb{Z}^N}} C_{\xi_1,\ldots,\xi_n}  \frac{1}{n}\sum_{j=1}^n|D^{\xi_j}_\varepsilon z(0)|^p.
\end{align*}
Hence,  by choosing
\begin{align*}
C^{0,\xi}_{\varepsilon,\delta} =  
\underset{\varepsilon|\xi_i|_\infty>\delta\text{ for some } i}{\sum_{\xi\in \{\xi_1,\ldots,\xi_n\} \subset (\mathbb{Z}^N)^n}} \frac{1}{n}C_{\xi,\ldots,\xi_n} , \quad  C^{j,\xi}_{\varepsilon,\delta} =0, j \neq 0
\end{align*}
(H4) follows. Setting
\begin{align*}
C^{0,\xi}_\varepsilon = \sum_{\xi\in \{\xi_1,\ldots,\xi_n\} \subset (\mathbb{Z}^N)^n} \frac{1}{d}C_{\xi,\ldots,\xi_n} , \quad  C^{j,\xi}_{\varepsilon} =0, j \neq 0
\end{align*}
we have that  $C^{j,\xi}_\varepsilon$ satisfies (\ref{Assumptions on Cxi}) and we have
\begin{align*}
\phi_i^\varepsilon(\{z_{j}\}_{j \in Z_\varepsilon(\Omega_i)}) \leq \sum_{\xi \in \mathbb{Z}^N} C^{0,\xi}_\varepsilon ( |D^{\xi}_\varepsilon z(0)|^p +1).
\end{align*}
Note that for all cut-off functions $\psi$ and for all $z,w \in \mathcal{A}_\varepsilon(\Omega;\mathbb{R}^n)$ we have
\begin{align}\label{finitedifferenz}
D^\xi_\varepsilon (\psi z + (1-\psi)w) = \psi(i) D^\xi_\varepsilon z(i) + (1-\psi(i))D^\xi_\varepsilon w(i) + D^\xi_\varepsilon \psi(i) (z(i+\varepsilon\xi)-w(i+\varepsilon\xi))
\end{align}
 and hence (H5) follows by using the convexity of $|\cdot|^p$, $0\leq\psi \leq 1 $ and noting that
 \begin{align} \label{maxbound}
| D^\xi_\varepsilon \psi(i)| \leq \max_{n \in \{1,\ldots,N\}} \sup_{k \in Z_\varepsilon(\Omega)} |D^{e_n}_\varepsilon \psi(k)|.
 \end{align}
 A particular example could be $g^{\varepsilon}_{e_1,e_2}(z)=|z|$ and $g^{\varepsilon}_{\xi_1,\xi_2}(z)=0$ otherwise. More general our Theorems also apply to the case where we take functions $g$ of minors of $\left(D^{\xi_1}_\varepsilon z(0),\ldots, D^{\xi_n} z(0)\right)$ as long as $g$ satisfies appropriate bounds.
\subsection{The linearization of the Lennard-Jones potential}
We assume $N=d=3$. Our result is applicable to show an integral representation if the potential $\phi_i^\varepsilon$ is the linearization of the Lennard-Jones potential, where the Lennard-Jones potential, pictured in Fig. \ref{LJ Potential}, is defined by (up to renormalization)
\begin{align*}
V(r) = \frac{1}{r^{12}}-\frac{2}{r^6}.
\end{align*}
For $\Omega \subset \mathbb{R}^3$ open and smooth we define $E_\varepsilon : L^2(\Omega;\mathbb{R}^3) \to [0,+\infty] $ by
\begin{align*}
E_\varepsilon(u) =\begin{cases}
\sum_{i,j \in Z_\varepsilon(\Omega)} \varepsilon^3 V''\Big(\left| \frac{i-j}{\varepsilon}\right|\Big) \left| \frac{u_i-u_j}{\varepsilon}\right|^2 &\text{if } u \in \mathcal{A}_\varepsilon(\Omega;\mathbb{R}^3)\\
+\infty &\text{otherwise.}
\end{cases}
\end{align*}
In fact heuristically $E_\varepsilon$ can be obtained by linearizing the Lennard-Jones Energy defined by
\begin{align*}
E^{LJ}_\varepsilon(u) =\begin{cases}
\sum_{i,j \in Z_\varepsilon(\Omega)} \varepsilon^3 V\Big(\left| \frac{u_i-u_j}{\varepsilon}\right|\Big) &\text{if } u \in \mathcal{A}_\varepsilon(\Omega;\mathbb{R}^3)\\
+\infty &\text{otherwise,}
\end{cases}
\end{align*}
where the set of admissible deformations $u$ should be close to the identity (neglecting the linear term in the expansion by the assumption that $u(i)=i$ is an equilibrium point). The term 
\begin{align*}
\tilde{\phi}_i^\varepsilon(\{u_{j+i}\}_{j \in Z_\varepsilon(\Omega_i)})=\sum_{j \in Z_\varepsilon(\Omega)} V''\Big(\left| \frac{i-j}{\varepsilon}\right|\Big) \left| \frac{u_i-u_j}{\varepsilon}\right|^2
\end{align*}
may not be positive in general, due to the long-range part of the potential. 
\begin{figure}[htp]
\centering\includegraphics{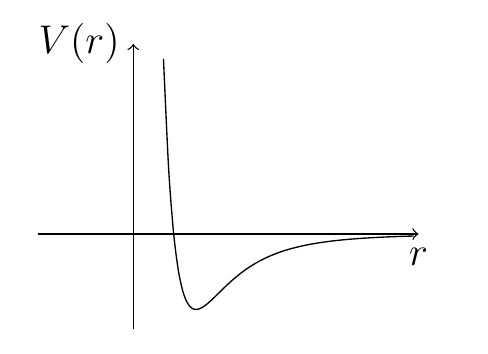}
\caption{The Lennard-Jones potential}
\label{LJ Potential}
\end{figure}
We may regroup the interactions of $\tilde{\phi}_i^\varepsilon$ such that we have a positive potential satisfying the assumptions of our main theorem. For every $\xi \in \mathbb{Z}^3, i \in \mathbb{Z}^N$ we choose a path $\gamma_\xi^i = (i_k)_{h=1}^{||\xi||_1+1}\subset \mathbb{Z}^3$ by defining 
\begin{align*}
i^\xi_{||\xi||_1+1} = i +\xi, i^\xi_1 = i, i_{h+1}^\xi =i_{h}^\xi+e_{n(h)}, e_{n(h)} = \mathrm{sign}(\xi_k) e_k \text{ if } 1+\sum_{n=1}^{k-1} | \xi_n | \leq h \leq \sum_{n=1}^{k} | \xi_n |.
\end{align*}
For this path it holds
\begin{align*}
|D_1^\xi u(i)|^2 \leq  \frac{3}{||\xi||_1}\sum_{h=1}^{||\xi||_1} |D^{e_n(h) }_1 u(i_h)|^2.
\end{align*}
For $\xi \in \mathbb{Z}^3\setminus \{\pm e_1,\pm e_2,\pm e_3\}$ we define $f^\xi_i : \mathcal{A}_1(\mathbb{R}^3;\mathbb{R}^3) \to [0,\infty)$ by
\begin{align*}
f^\xi_i (\{u_j\}_{j \in \mathbb{Z}^3}) = V''(|\xi|)\Big( -|D_1^\xi u(i)|^2 +  \frac{3}{||\xi||_1}\sum_{h=1}^{||\xi||_1} |D^{e_{n(h)} }_1 u(i_h)|^2\Big),
\end{align*}
and for $v \in  \{\pm e_1,\pm e_2,\pm e_3\}$ we define $f^{ v }_i: \mathcal{A}_1(\mathbb{R}^3;\mathbb{R}^3) \to \mathbb{R}$ 
\begin{align*}
f^{ v }_i (\{u_j\}_{j \in \mathbb{Z}^3}) = \Big(V''(1)- \sum_{j \in \mathbb{Z}^3}\underset{i = i_h \in \gamma_j^\xi, e_{n(h)}=v}{\sum_{\xi \in \mathbb{Z}^3}} \frac{3V''(|\xi|)}{||\xi||_1}\Big)|D^{v }_1 u(i)|^2.
\end{align*}
Moreover, we define $\phi_i^k : \mathcal{A}_1(\mathbb{R}^3;\mathbb{R}^3) \to \mathbb{R} $ by
\begin{align*}
\phi^{ k }_i (\{u_j\}_{j \in \mathbb{Z}^3}) = \sum_{ |\xi|_\infty \leq k} f^\xi_i(\{u_j\}_{j \in \mathbb{Z}^3}) 
\end{align*}
and $\phi_i : \mathcal{A}_1(\mathbb{R}^3;\mathbb{R}^3) \to \mathbb{R} $ by 
\begin{align*}
\phi_i (\{u_j\}_{j \in \mathbb{Z}^3}) = \sum_{\xi \in \mathbb{Z}^3} f^\xi_i(\{u_j\}_{j \in \mathbb{Z}^3}).
\end{align*}
We need to check that 
\begin{align}\label{coercivity}
f^{ v }_i (\{u_j\}_{j \in \mathbb{Z}^N}) \geq  c |D^{v }_1 u(i)|^2
\end{align}
for some constant $c>0$, $v \in   \{\pm e_1,\pm e_2,\pm e_3\}$ and that $\phi^k_i, \phi_i$ satisfy (H1)--(H3) and ($\text{H}_p$4)--($\text{H}_p$7). Note that for $u^\varepsilon(j) = \frac{u(\varepsilon j)}{\varepsilon}$ it holds
\begin{align*}
\sum_{i \in Z_\varepsilon(\mathbb{R}^3)}\phi_{\frac{i}{\varepsilon}}(\{u^\varepsilon_j\}_{j \in \mathbb{Z}^3})= \sum_{i \in Z_\varepsilon(\mathbb{R}^3)}\tilde{\phi}_{ i}^\varepsilon(\{u_j\}_{j  \in Z_\varepsilon(\mathbb{R}^3)}) .
\end{align*}
By the definition of $\phi_i^k, \phi_i$ it is clear, that (H1), (H2) holds. To prove (H3) we have that $\phi_i^\xi \geq 0$ for all $\xi \in \mathbb{Z}^3 \setminus \{\pm e_1, \pm e_2, \pm e_3\}$ and by Fubini's Theorem we have that
\begin{align}\label{Cxicomp1}
\sum_{j \in \mathbb{Z}^3}\underset{i = i_h \in \gamma_j^\xi, e_{n(h)}=v}{\sum_{\xi \in \mathbb{Z}^3,|\xi| >1}} \frac{3V''(|\xi|)}{||\xi||_1} = \sum_{\xi \in \mathbb{Z}^3,|\xi|>1} \# N_{i,v}^\xi \frac{3V''(|\xi|)}{||\xi||_1},
\end{align}
where 
$ \displaystyle
N_{i,v}^\xi = \Big\{j \in \mathbb{Z}^3 : \, \exists h \in \{1,\ldots,|\xi_1|\}\text{ such that } i = j_h \in \gamma^\xi_j \text{ and } e_{n(h)}= v  \Big\}.
$ Note that for $\xi \in \mathbb{Z}^3$ such that $\langle \xi, v\rangle >0$ we have  $\# N_{i,v}^\xi \leq ||\xi||_1$ and $\# N_{i,v}^\xi=0$ otherwise. Hence, using the monotonicity of $V''(r)$ for $r \geq \sqrt{2} $ and the fact that $||\xi||_\infty \leq ||\xi||_2$ and using the fact that $\#\{\xi \in \mathbb{Z}^3 : ||\xi||_\infty = k \} = 3k^2-3k+1 $ , we obtain
\begin{align}\label{Cxicomp2}
-\nonumber\underset{|\xi|>1}{\sum_{\xi \in \mathbb{Z}^3}} \# N_{i,v}^\xi \frac{3V''(|\xi|)}{||\xi||_1}  &\leq -\underset{\langle \xi, v\rangle >0}{\sum_{\xi \in \mathbb{Z}^3,|\xi|>1}} 3V''(|\xi|) = -12 V''(\sqrt{2}) - 3\sum_{ k=2}^\infty \sum_{||\xi|_\infty =k} V''(|\xi|)(3k^2-3k+1) \\&\leq  -12 V''(\sqrt{2}) - 3\sum_{ k=2}^\infty V''(k)(3k^2-3k+1) < V''(1).
\end{align}
Hence,  we obtain (\ref{coercivity}) and with that (H3). ($\text{H}_p$4) and ($\text{H}_p$7) follow from the definition of $\phi_i^k$ and $\phi_i$. Setting 
\begin{align}\label{Cj}
C^{j,e_n} = \begin{cases} V''(1) &\text{if }  j=0 \\
\displaystyle\underset{j=i_h,e_{n(h)}=e_n}{\sum_{\xi \in \mathbb{Z}^3,|\xi|>1}} \frac{3 V''(|\xi|)}{||\xi||_1} &\text{otherwise,}
\end{cases}
\end{align}
and $C^{j,\xi}=0$ if $|\xi|>1$. Using (\ref{Cxicomp1}) and (\ref{Cxicomp2}) we obtain (\ref{Assumptions on Cjk}) and
\begin{align*}
\phi_i^k(\{\psi_{j}z_{j} + (1-\psi_{j})w_{j}\}_{j \in \mathbb{Z}^N}) \leq R^k_i(z,w,\psi),
\end{align*}
with $R_i^k$ defined in ($\text{H}_p5$) with $C^{j,\xi}$ defined by (\ref{Cj}). By the non-negativity of $\phi_i^k$ it follows ($\text{H}_p5$). 
Setting 
\begin{align}\label{Cjk}
C^{j,e_n}_k = 2\underset{j=i_h,e_{n(h)}=e_n}{\sum_{\xi \in \mathbb{Z}^3,||\xi||_\infty>k}} \frac{3 V''(|\xi|)}{||\xi||_1}
\end{align}
and $C^{j,\xi}_k=0$ if $|\xi|>1$, using (\ref{Cxicomp1}) and (\ref{Cxicomp2}) we obtain (\ref{Assumptions on Cxik}). We have that
\begin{align*}
|\phi_i^{k_1}(\{z_{j}\}_{j \in \mathbb{Z}^N}) - \phi_i^{k_2}(\{z_{j}\}_{j \in \mathbb{Z}^N})|& = \sum_{\xi \in \mathbb{Z}^3 \cap (Q_{k_2} \setminus Q_{k_1})} f_i^\xi(\{z_j\}_{j \in \mathbb{Z}^3}) \\&\leq \sum_{\xi \in \mathbb{Z}^3 \cap (Q_{k_2} \setminus Q_{k_1})}V''(|\xi|) \frac{3}{||\xi||_1}\sum_{h=1}^{||\xi||_1} |D^{e_{n(h)} }_1 z(i_h^\xi)|^2 \\&\leq \sum_{n=1}^3\sum_{j \in \mathbb{Z}^3 \cap Q_{k_2}} \underset{j=i_h,e_{n(h)} \in \{\pm e_n\}}{\sum_{\xi \in \mathbb{Z}^3,||\xi||_\infty>k_1}} \frac{3 V''(|\xi|)}{||\xi||_1} |D^{e_{n} }_1 z(j)|^2 \\&\leq \underset{j+\xi \in \mathbb{Z}^3 \cap Q_{k_2}}{\sum_{j,\xi \in \mathbb{Z}^3 \cap Q_{k_2} }} C^{j,\xi}_{k_1}  |D^{\xi}_1 z(j)|^2 
\end{align*}
and hence we obtain ($\text{H}_p$6).
Applying Theorem \ref{Homogenization} we obtain the $\Gamma$-convergence of $E_\varepsilon$ to a functional $E : L^p(\Omega;\mathbb{R}^3) \times \mathcal{A}(\Omega) \to [0,+\infty] $ given by
\begin{align*}
E(u,A) = \int_A f_{\mathrm{hom}}(\nabla u) \mathrm{d}x, 
\end{align*}
where $f_{\mathrm{hom}} : \mathbb{R}^{3 \times 3} \to [0,+\infty) $ is given by
\begin{align*}
f_{\mathrm{hom}}(M) = \lim_{L \to \infty} \frac{1}{L^N} \inf \Big\{\sum_{i \in \mathbb{Z}^N \cap Q_L} \phi_{i}(\{z_{j+i}\}_{j \in \mathbb{Z}^N }) : z \in \mathcal{A}^{M,m}_1(Q_L;\mathbb{R}^n) \Big\}.
\end{align*}
\subsection{Pair interactions: the Alicandro-Cicalese theorem}\label{ACt} The compactness theorem can be applied to the special case of pair potentials where $\phi_i^\varepsilon$ takes only into account the pair interactions of that point with every other point $j \in Z_\varepsilon(\Omega)$, that means it is of the form
\begin{align*}
\phi_i^\varepsilon(\{z_{j+i}\}_{j\in Z_\varepsilon(\Omega_i)}) = \underset{i+\varepsilon\xi \in Z_\varepsilon(\Omega)}{\sum_{\xi \in \mathbb{Z}^N}}f^\xi_\varepsilon(i, D^\xi_\varepsilon z(i))
\end{align*}
with $f^\xi_\varepsilon \geq 0$ satisfying
\begin{itemize}
\item[(i)] $f^{e_n}_\varepsilon(i,z) \geq c(|z|^p-1) $ for all $i \in Z_\varepsilon(\Omega)$, $z \in \mathbb{R}^n,\varepsilon >0$ and $n \in \{1,\ldots,N\}$.
\item[(ii)] $f^\xi_\varepsilon(i,z) \leq c^\xi_\varepsilon(|z|^p +1)$ for all $i \in Z_\varepsilon(\Omega)$, $z \in \mathbb{R}^n,\varepsilon >0$ and $\xi \in \mathbb{R}^N$, where
\begin{align} \label{properties Cxi1}
&\limsup_{\varepsilon \to 0}\sum_{\xi \in \mathbb{Z}^N} c^\xi_\varepsilon < +\infty \\ \label{properties Cxi2}\forall \, \delta> 0 \, \exists M_\delta >0 \text{ such that } &\limsup_{\varepsilon \to 0} \sum_{|\xi| > M_\delta} c^\xi_\varepsilon <\delta.
\end{align}
\end{itemize}
(H1) follows since for each $\xi \in \mathbb{R}^N, i \in Z_\varepsilon(\Omega)$ the interaction depend only on $D^\xi_\varepsilon z$. (H2) follows from (\ref{properties Cxi1}) and (ii). (H3) follows from (i). (H4) follows if we choose 
\begin{align*}
C_{\varepsilon,\delta}^{i,\xi} = \begin{cases}  c^\xi_\varepsilon &\varepsilon|\xi|_\infty \geq \delta, i =0 \\   0 &i \neq 0.
\end{cases}
\end{align*}
Let $z,w \in \mathcal{A}_\varepsilon(\Omega;\mathbb{R}^n)$ such that $z(j)=w(j) $ in $Z_\varepsilon(Q_\delta(i))$. 
Then, using the positivity of $f^\xi_\varepsilon$ and (ii), we obtain
\begin{align*}
\phi_i^\varepsilon(\{z_{j}\}_{j\in Z_\varepsilon(\Omega_i)}) &= \underset{i+\varepsilon\xi \in Z_\varepsilon(\Omega_i)}{\sum_{\xi \in \mathbb{Z}^N}}f^\xi_\varepsilon(0, D^\xi_\varepsilon z(0)) =  \underset{\varepsilon\xi \in Z_\varepsilon(\Omega_i)}{\sum_{|\xi|_\infty\varepsilon \leq \delta}}f^\xi_\varepsilon(0, D^\xi_\varepsilon z(0)) + \underset{\varepsilon\xi \in Z_\varepsilon(\Omega_i)}{\sum_{|\xi|_\infty\varepsilon > \delta}}f^\xi_\varepsilon(0, D^\xi_\varepsilon z(0)) \\&\leq \underset{\varepsilon\xi \in Z_\varepsilon(\Omega_i)}{\sum_{|\xi|_\infty\varepsilon \leq \delta}}f^\xi_\varepsilon(0, D^\xi_\varepsilon w(0)) + \underset{\varepsilon\xi \in Z_\varepsilon(\Omega_i)}{\sum_{|\xi|_\infty\varepsilon > \delta}}c^\xi_\varepsilon( |D^\xi_\varepsilon z(0)|^p +1) \\&\leq \phi_i^\varepsilon(\{w_{j}\}_{j\in Z_\varepsilon(\Omega_i)}) +\sum_{j\in Z_\varepsilon(\Omega_i),\xi \in \mathbb{Z}^N}C^{j,\xi}_{\varepsilon,\delta}(|D^\xi_\varepsilon z(j)|^p +1)
\end{align*} 
and therefore (H4) follows. Setting  
\begin{align*}
C^{i,\xi}_{\varepsilon} = \begin{cases} c^\xi_\varepsilon &\text{if }i=0 \\
0 &\text{otherwise}.
\end{cases}
\end{align*}
we have that
\begin{align*}
\phi_i^\varepsilon(\{z_{j}\}_{j \in Z_\varepsilon(\Omega_i)}) \leq \underset{j+\varepsilon\xi \in Z_\varepsilon(\Omega_i)}{\sum_{j \in Z_\varepsilon(\Omega_i),\xi \in \mathbb{Z}^N}} C^{j,\xi}_\varepsilon |D^\xi_\varepsilon z(j)|^p.
\end{align*}
and again for a cut-off function $\psi$ and $z,w \in \mathcal{A}_\varepsilon(\Omega;\mathbb{R}^n)$ (H5) follows by using (\ref{finitedifferenz}), the convexity of $|\cdot|^p$ and (\ref{maxbound}).

\end{document}